\documentclass[11pt,a4paper]{amsart}
\usepackage[utf8]{inputenc}
\usepackage{lmodern}
\usepackage[T1]{fontenc}
\usepackage[english]{babel}
\usepackage{ifpdf}
\usepackage[left=1in, right=1in, top=1in, bottom=1in]{geometry}
\usepackage[dvipsnames]{xcolor}
\usepackage{hyperref}
\hypersetup{colorlinks=true, linkcolor=blue, citecolor=red}

\usepackage[foot]{amsaddr}
\usepackage{graphicx, amsmath, amsthm, amssymb, enumitem, mathrsfs} 
\usepackage{subcaption,caption}
\usepackage{booktabs}

\usepackage[capitalize]{cleveref}
\usepackage[square,numbers]{natbib}
\usepackage{soul}
\usepackage{multicol}
\usepackage{wrapfig}

\usepackage{tikz}
\usepackage{tikz-3dplot} 
\usepackage{pgfplots, pgfplotstable, enumitem}
\usepackage{epstopdf}
\usetikzlibrary{3d} 
\usetikzlibrary{arrows.meta}
\usepgfplotslibrary{fillbetween}
\usetikzlibrary{patterns, patterns.meta}
\usetikzlibrary{shapes.geometric}
\usepackage{pgffor}

\usepackage{fancyhdr} % Required package for custom headers and footers

\newcommand{\N}{\mathbb{N}}

\newcommand{\diam}{\operatorname{diam}}

\newcommand{\forman}[1]{\kappa_F\left({#1}\right)}

\newcommand{\avgdegree}{\rho}

\newcommand{\resistance}[1]{\kappa_R\left({#1}\right)}

\newtheorem{theorem}{Theorem}[section]
\newtheorem{corollary}[theorem]{Corollary}
\newtheorem{remark}[theorem]{Remark}
\newtheorem{definition}[theorem]{Definition}
\newtheorem{lemma}[theorem]{Lemma}
\newtheorem{proposition}[theorem]{Proposition}
\newtheorem{conjecture}[theorem]{Conjecture}

\newtheorem{example}[theorem]{Example}

\usepackage{xcolor, soul}

 
\date{\textbf{\texttt{Keywords:}} discrete curvature, convex polytopes, diameter, Forman--Ricci curvature, effective resistance, graph Laplacians
\\
\hspace*{.225in}\textbf{\texttt{MSC2020:}} 52B05, 52B11, 05C50, 05C10}

\title{Discrete Curvatures and Convex Polytopes}
\author{Jes\'us A. De Loera$^1$}
\author{Jillian Eddy$^1$}
\author{Sawyer J. Robertson$^2$}
\author{Jos\'e A. Samper$^3$}

\address{$^1$Department of Mathematics, UC Davis}
\address{$^2$Department of Mathematics, UC San Diego}
\address{$^3$Departamento de Matem\'{a}ticas, Pontificia Universidad Cat\'{o}lica de Chile}

\hypersetup{
	pdftitle={Discrete Curvatures and Convex Polytopes},
	pdfauthor={J. A. De Loera, J. Eddy, S. J. Robertson, J. A. Samper}
} 

\begin{document}
    \captionsetup[figure]{labelfont={bf},labelformat={default},labelsep=period,name={Figure}}

    \begin{abstract}
        We study Forman--Ricci and effective resistance curvatures on the skeleta of convex polytopes. Our guiding questions are: how frequently do polytopal graphs exhibit everywhere positive curvature, and what structural constraints does positivity impose? For Forman--Ricci curvature we derive an exact identity for the average edge curvature in terms of flag $f$-numbers and establish the existence of infinite families of Forman--Ricci-positive polytopes in every fixed dimension $d\ge 6$. We prove finiteness results in low dimension: there are only finitely many Forman--Ricci-positive $3$- and $4$-polytopes; for $d=5$ we show finiteness in the simplicial case, and conjecture its extension to $5$-polytopes more generally. For the resistance curvature $\kappa(v)$ we establish the existence of infinite families for all $d\ge 3$, and we provide a quantitative lower bound for $\kappa(v)$ in a simple $3$-polytope in terms of the lengths of the three $2$-faces incident to $v$. This bound leads to constructions of non-vertex-transitive, resistance-positive $3$-polytopes via $\Delta$-operations, and a degree-based obstruction showing that if each neighbor of $v$ has degree at most $d_v-2$, then $\kappa(v)\le 0$. Our results suggest that positive curvature on polytopal skeletons is rare and constrained. 
    \end{abstract}

	\maketitle

\section{Introduction}

Curvature of graphs and other discrete structures is a rapidly developing field rooted in deep questions concerning how well discrete models capture the geometric structure of continuous spaces. Although the notion of \emph{discrete curvature} may seem counterintuitive at first, there exist many notions of curvature on graphs and complexes which have been shown to satisfy discrete analogues of well-known results from differential geometry. Examples include Bonnet-Myers-type theorems relating curvature to diameter bounds (see, e.g.,~\cite[Thm. 1]{devriendt2024graph},~\cite[Thm. 6.3]{forman}) and Lichnerowicz-type bounds relating curvature to Laplacian eigenvalues (see, e.g.,~\cite[Thm. 4.2]{lin2011ricci},~\cite[Thm. 3]{steinerberger2023curvature}).
Additionally, discrete curvatures have been used in been applied data science and the analysis of networks, demonstrating their versatility and importance.
For example, Weber and others~\cite{Weber2017,Fesseretal2024} connected Forman--Ricci curvature to the analysis of complex networks, while Ollivier and others related the curvature of Markov chains to their mixing rates and spectral gaps, showing that positive curvature ensures fast convergence (see e.g.,~\cite{Ollivier2009,munch2022mixingtimeexpansionnonnegatively} and references therein). 

Meanwhile, in polyhedral geometry, many longstanding open questions remain which concern the very quantities investigated in the theory of discrete curvatures. The \emph{Hirsch conjecture} (see~\cite{Ziegler}), for example, predicted that the largest diameter $f(d,n)$ of polytope of dimension $d\geq 1$ defined by no more than $n\ge 1$ linear inequalities satisfies $f(n,d)\le n-d$. The Hirsch conjecture was disproved in general by Santos~\cite{santos2012counterexample}, but variations of the conjecture remain open and of great interest to the community. Another example is \emph{Barnette's conjecture} (see~\cite{Barnette1969Conjecture5}), which hypothesizes that every cubic bipartite 3-dimensional polyhedral graph is Hamiltonian. It was shown recently by Devriendt~\cite{devriendt2025graphs} that Hamiltonian graphs have positive curvature with respect to a weighted variant of resistance curvature.

It is therefore natural to consider the properties of discrete curvatures within the category of convex polytopes and, in particular, their graphs (i.e., their $1$-skeleta). Little effort has been made in this research direction and we are not aware of any prior work. In this article, we investigate the discrete curvature of polytopes with emphasis on \emph{Forman--Ricci curvature} (see~\cref{def:forman-curvature}) and \emph{effective resistance curvature} (see~\cref{defn:curvature-resistance}). In both cases, we are interested in the basic question of \emph{how many positively curved combinatorial types of $d$-polytopes exist for various values of the dimension $d$}. Somewhat surprisingly, in both cases, we show that positive curvature polyhedra appear to be rare. We state our contributions later. 
\subsection{Related work}

Forman--Ricci curvature, introduced for general cell complexes~\cite{forman}, specializes to an edge-based invariant on graphs that is local and computationally cheap. Subsequent work has seen the theory develop further in various directions: Watanabe~\cite{Watanabe2019} established a Gauss-Bonnet-type theorem for graphs and $2$-complexes; Bloch~\cite{Bloch2014} analyzed structural limitations of the edge-only definition in dimension $2$ and proposed a poset-theoretic extension that restores a Gauss-Bonnet analogue and clarifies the failure of ubiquitous negativity on surfaces; Jost and M{\"u}nch~\cite{JostMuench2021} characterized lower bounds of Forman--Ricci curvature via the contractivity of the Hodge-Laplacian semigroup and related (optimized) 
Forman--Ricci and Ollivier curvatures, yielding refined diameter bounds and a bridge to heat semigroup techniques. Extensions of this notion to weighed graphs~\cite{Sreejith2016}, directed graphs~\cite{Sreejith2016}, and hypergraphs~\cite{Leal2021Hypergraph} have also been considered.

Resistance curvature, on the other hand, is comparatively newer and continues 
to be an active topic of research. Effective resistance (see~\cite{KleinRandic1993ResistanceDistance}), more generally, is a metric on the vertices of a graph and is related to the simple random walk on the graph~\cite{Tetali1991RandomWalksResistance,DoyleSnell1984RWEN}, spanning trees, and graph sparsification~\cite{SpielmanSrivastava2011Sparsification}. Using resistance as the basis for a notion of curvature was originally proposed by Devriendt and Lambiotte~\cite{DL}, and was followed shortly thereafter by a closely related notion by Devriendt, Ottolini, and Steinerberger~\cite{devriendt2024graph}. Subsequent work by Devriendt considered a relaxed notion of positive resistance curvature for graphs~\cite{devriendt2025graphs} and its connections to combinatorial properties of graphs satisfying this condition.

\subsection{Notation and mathematical background}
We follow the notation and conventions of the classical books~\cite{Grunbaum,Ziegler}.
A \emph{polytope} $P\subseteq\mathbb{R}^d$ is the convex hull of a finite collection of points. We do not consider nonconvex polytopes in this article. 
The \emph{dimension} of $P$ is the dimension of the smallest affine subspace containing it; and the \emph{codimension} is given by $d$ minus the dimension of $P$. A \emph{face} $Q\subseteq P$ is any subset of $P$ for which there exists a linear functional $\ell: \mathbb{R}^d\to \mathbb{R}$ which is constant on $Q$ and which satisfies
    \begin{align*}
        \max_{\vec{x}\in P}\ell(x) = \max_{\vec{x}\in Q}\ell(x).
    \end{align*}
Any face of a polytope is a polytope, and has a well defined dimension. Faces of dimension $0$ are called \emph{vertices} and faces of dimension $1$ are called \emph{edges} of $P$. The \emph{face lattice} of $P$ consists of all the faces of $P$ ordered by inclusion. We say that two polytopes are \emph{combinatorially equivalent} if they have isomorphic face lattices. In this article we focus on polytopes up to combinatorial equivalence.

Let $P$ be a $d$-dimensional polytope. For each $0\leq k \leq d-1$, we denote by $\mathcal{F}_k = \mathcal{F}_k(P)$ the collection of $k$-dimensional faces of $P$. We write $f_k = f_k(P)$ to refer to the cardinality of $\mathcal{F}_k$. If $0\le i <j\le d-1$ we denote by $f_{ij}$ the number of pairs $(F,G)$ with $F\in \mathcal{F}_i$, $G\in \mathcal{F}_j$ and $F\subseteq G$. The \emph{$k$-skeleton} of $P$ consists of the collection of all faces of dimension at most $k$. The \emph{graph} of $P$ is the combinatorial graph $G=G(P) = (\mathcal{F}_0(P) =: V(P), \mathcal{F}_1(P) =:E(P) )$, i.e., the 1-skeleton of $P$. Note that in general polytopes are not characterized by their graphs: polytopes whose graph is isomorphic to the complete graph are known as \emph{neighborly} and are abundant (see, e.g.,~\cite[Ch. 8]{Ziegler}). In general, the graph of a $d$-dimensional polytope is known to be $d$-vertex-connected (Balinski's theorem), and in dimension $3$, the graphs of $3$-polytopes are characterized combinatorially as exactly those graphs which are planar and 3-vertex-connected (Steinitz's theorem).

If $e\in E(P)$ is any edge, we denote by $\mathcal{F}\uparrow(e)\subseteq\mathcal{F}_2(P)$ the collection of $2$-faces of $P$ that contain $e$ and by $\mathcal{F}\downarrow(e)\subseteq\mathcal{F}_0(P)$ the set of vertices of $P$ contained in $e$.

\begin{definition}\label{def:parallel-neighbors}
    Let $P$ be a polytope and $e, e'\in E(P)$ fixed edges. We say that $e$ and $e'$ are {\normalfont parallel neighbors} if one of the following statements holds:
    \begin{itemize}
        \item[\textit{i)}] $\mathcal{F}\downarrow(e) \cap \mathcal{F}\downarrow(e')\not= \emptyset $, but $\mathcal{F}\uparrow(e) \cap \mathcal{F}\uparrow(e')= \emptyset$, i.e., if $e$ and $e'$ share a vertex, but are not contained in a common two face. 

        \item[\textit{ii)}] $\mathcal{F}\uparrow(e) \cap \mathcal{F}\uparrow(e') \not= \emptyset$, but $\mathcal{F}\downarrow(e)\cap \mathcal{F}\downarrow(e')= \emptyset$, i.e., $e$ and $e'$ are vertex disjoint edges that are contained in a two dimensional face. 
    \end{itemize}
    The collection of parallel edges of $e$ is denoted by $\mathcal{E}(e)$.
\end{definition}

\begin{figure}[t!]
    \centering
    \begin{subfigure}[t]{0.48\textwidth}
        \centering
            \begin{tikzpicture}
        \coordinate (A) at (1,0) {};
        \coordinate (B) at (0.62348, 0.7818) {};
        \coordinate (C) at (-0.2225,0.9749) {};
        \coordinate (D) at (-0.9,0.433) {};
        
        \coordinate (E) at (-0.9,-0.433) {};
        \coordinate (F) at (-.2225,-0.9749) {} ;
        \coordinate (G) at (0.62348,-0.7818) {} ;
    
        \fill[pattern=crosshatch dots, pattern color=gray!50] (A) -- (B) -- (C) -- (D) -- (E) -- (F) -- (G) -- cycle;
        
        \path[draw=red, line width=0.5mm] (A) edge node [near end, right] {$e$} (B);
        \path (B) edge node {} (C);
        \path[draw=blue, line width=0.5mm] (C) edge node {} (D);
        \path[draw=blue, line width=0.5mm] (D) edge node {} (E);
        \path[draw=blue, line width=0.5mm] (E) edge node {} (F);   
        \path[draw=blue, line width=0.5mm] (F) edge node {} (G);   
        \path (G) edge node {} (A);
    \end{tikzpicture}
    \hspace{1cm}
        \begin{tikzpicture}
        \coordinate (A) at (1,0) {};
        \coordinate (B) at (0.62348, 0.7818) {};
        \coordinate (C) at (-0.2225,0.9749) {};
        \coordinate (D) at (-0.9,0.433) {};
        
        \coordinate (E) at (-0.9,-0.433) {};
        \coordinate (F) at (-.2225,-0.9749) {} ;
        \coordinate (G) at (0.62348,-0.7818) {} ;
        \coordinate (H) at (0,0) {} ;
        
        \fill[pattern=crosshatch dots, pattern color=gray!50] (A) -- (B) -- (C) -- (D) -- (E) -- (F) -- (G) -- cycle;
            
        \path[draw=red,line width=0.5mm] (A) edge node [midway, above] {$e$} (H);
        \path (B) edge node {} (H);
        \path[draw=blue, line width=0.5mm] (C) edge node {} (H);
        \path[draw=blue, line width=0.5mm] (D) edge node {} (H);
        \path[draw=blue, line width=0.5mm] (E) edge node {} (H);   
        \path[draw=blue, line width=0.5mm] (F) edge node {} (H);   
        \path (G) edge node {} (H);
    \end{tikzpicture}
        \caption{}\label{subfig:parallel-neighbors}
    \end{subfigure}
    \begin{subfigure}[t]{0.48\textwidth}
        \centering
        % SQUARE CUPOLA
% \begin{figure}[h]
\begin{tikzpicture}[x={(0.866cm,0.3cm)}, y={(-0.866cm,0.3cm)}, z={(0cm,1cm)}, scale=1.4]
% [circle, fill=blue, inner sep=0pt,minimum size=5pt]

\coordinate (P1) at (0.5,0.5,0.71); % A
\coordinate (P2) at (-0.5,0.5,0.71); % B
\coordinate (P3) at (-0.5,-0.5,0.71); % D
\coordinate (P4) at (0.5,-0.5,0.71); % C

\coordinate (D1) at (1.21, 0.5, 0); % E
\coordinate (D2) at (0.5, 1.21, 0); % I
\coordinate (D3) at (-0.5, 1.21, 0); % J
\coordinate (D4) at (-1.21, 0.5, 0); % F
\coordinate (D5) at (-1.21, -0.5, 0); % H
\coordinate (D6) at (-0.5, -1.21, 0); % L
\coordinate (D7) at (0.5, -1.21, 0); % K
\coordinate (D8) at (1.21, -0.5, 0); % G

% TOP FACE
%\draw[fill opacity=0.7, fill = gray, line width=0.5mm] (P1) -- (P2) -- (P3) -- (P4) -- cycle; % ABDC
\draw[pattern=crosshatch dots, pattern color=gray!50] (P1) -- (P2) -- (P3) -- (P4) -- cycle; % ABDC

% TRIANGLES
% \draw[fill opacity=0.2,fill=pur] (D1) -- (P1) -- (D2) -- cycle; % EAI BACK
\draw[pattern=crosshatch dots, pattern color=gray!50] (D3) -- (P2) -- (D4) -- cycle; % JBF LEFT
\draw[pattern=crosshatch dots, pattern color=gray!50] (D5) -- (P3) -- (D6) -- cycle; % FRONT
\draw[pattern=crosshatch dots, pattern color=gray!50] (D7) -- (P4) -- (D8) -- cycle; % RIGHT

% SQUARES
% \draw[fill opacity=0.1,fill=greeo] (D2) -- (P1) -- (P2) -- (D3) -- cycle; % IABJ
\draw[pattern=crosshatch dots, pattern color=gray!50] (D4) -- (P2) -- (P3) -- (D5) -- cycle; % FBDH
\draw[pattern=crosshatch dots, pattern color=gray!50] (D6) -- (P3) -- (P4) -- (D7) -- cycle; %CTMIP
% \draw[fill opacity=0.1,fill=greeo] (D8) -- (P4) -- (P1) -- (D1) -- cycle; %BPHLS

% FORMAN CURVATURE

% BLUE = NEGATIVE
% RED = POSITIVE
% BLACK = 0

% TOP
\draw[line width=0.5mm, color = black] (P1) -- (P2) node [midway, above] {0};
\draw[line width=0.5mm, color = black] (P2) -- (P3) node [midway, below] {0};
\draw[line width=0.5mm, color = black] (P3) -- (P4) node [midway, below] {0};
\draw[line width=0.5mm, color = black] (P4) -- (P1) node [midway, above] {0};

% BOTTOM 
\draw[line width=0.5mm, color = blue] (D3) -- (D4) node [midway, left] {-1};
\draw[line width=0.5mm, color = blue] (D4) -- (D5) node [midway, below] {-2};
\draw[line width=0.5mm, color = blue] (D5) -- (D6) node [midway, below] {-1};
\draw[line width=0.5mm, color = blue] (D6) -- (D7) node [midway, below] {-2};
\draw[line width=0.5mm, color = blue] (D7) -- (D8) node [midway, right] {-1};

\draw[line width=0.5mm, dotted] (D8) -- (D1) -- (D2) -- (D3);

% TRIANGLES 
\draw[line width=0.5mm, color = red] (D3) -- (P2) node [midway, above] {2};
\draw[line width=0.5mm, color = red] (P2) -- (D4) node [midway, right] {2};

\draw[line width=0.5mm, color = red] (D5) -- (P3) node [midway, left] {2};
\draw[line width=0.5mm, color = red] (P3) -- (D6) node [midway, right] {2};

\draw[line width=0.5mm, color = red] (D7) -- (P4) node [midway, left] {2};
\draw[line width=0.5mm, color = red] (P4) -- (D8) node [midway, above] {2};

% RESISTANCE CURVATURE

%\node[circle, fill=blue, inner sep=0pt,minimum size=5pt] at (0,0.85,1.38) {}; 

% BLUE = NEGATIVE
% RED = POSITIVE

\node[circle, fill=black, inner sep=0pt,minimum size=5pt] at (0.5,0.5,0.71) {}; % A
\node[circle, fill=black, inner sep=0pt,minimum size=5pt] at (-0.5,0.5,0.71) {}; % B
\node[circle, fill=black, inner sep=0pt,minimum size=5pt] at (-0.5,-0.5,0.71) {}; % D
\node[circle, fill=black, inner sep=0pt,minimum size=5pt] at (0.5,-0.5,0.71) {}; % C

%\node[circle, fill=red, inner sep=0pt,minimum size=5pt] at (1.21, 0.5, 0) {}; % E
%\node[circle, fill=red, inner sep=0pt,minimum size=5pt] at (0.5, 1.21, 0) {}; % I
\node[circle, fill=black, inner sep=0pt,minimum size=5pt] at (-0.5, 1.21, 0) {}; % J
\node[circle, fill=black, inner sep=0pt,minimum size=5pt] at (-1.21, 0.5, 0) {}; % F
\node[circle, fill=black, inner sep=0pt,minimum size=5pt] at (-1.21, -0.5, 0) {}; % H
\node[circle, fill=black, inner sep=0pt,minimum size=5pt] at (-0.5, -1.21, 0) {}; % L
\node[circle, fill=black, inner sep=0pt,minimum size=5pt] at (0.5, -1.21, 0) {}; % K
\node[circle, fill=black, inner sep=0pt,minimum size=5pt] at (1.21, -0.5, 0) {}; % G
\end{tikzpicture}
% \caption{$3$-dimensional polytope ``Square Cupola" with edges labeled according to their Forman curvature.}
% \end{figure}
        \caption{}\label{subfig:forman-curvature}
    \end{subfigure}%
    ~
    \caption{\emph{(a)} Parallel neighbors (blue) of an edge $e$ (red) in the case where $e$ is an edge of a heptagon (left), and in the case where $e$ is adjacent to a vertex of degree 7 (right). \emph{(b)} A $3$-dimensional square cupola polytope with edges labeled according to their Forman--Ricci curvature.}\label{fig:parallel-forman}
\end{figure}
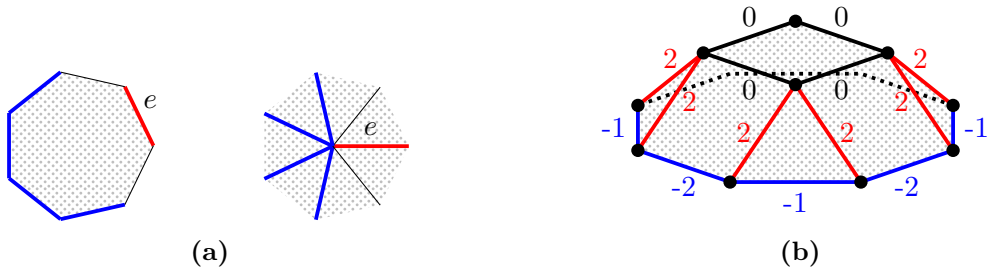

\begin{definition}\label{def:forman-curvature}
    Let $P$ be a polytope and let $e\in E(P)$ be any fixed edge. The {\normalfont Forman--Ricci curvature} of $e$, denoted $\forman{e}$, is given by
        \begin{align*}
            \forman{e} := |\mathcal{F}\uparrow(e)| + 2 - |\mathcal{E}(e)|.
        \end{align*}
\end{definition}

We illustrate~\cref{def:parallel-neighbors} and~\cref{def:forman-curvature} in~\cref{subfig:parallel-neighbors} and~\cref{subfig:forman-curvature}, respectively.~\cref{def:forman-curvature} originally appeared (for cell complexes) in 2003 in a work of Forman (see~\cite{forman}) and is known in the literature by this name. The same paper contains a Bonnet-Myers-type diameter bound, which is set up as follows. The distance between two vertices $v,v'\in\mathcal{F}_0$, denoted $d(v, v')$, is the length of any shortest path from $v$ to $v'$ in $G(P)$. The diameter of $G$ is the largest distance between a pair of vertices and is denoted by $\diam(G)$. The degree of a vertex $v$, denoted $d_v$, is the number edges incident to $v$. A $d$-polytope $P$ is said to be \emph{simple} if its graph is $d$-regular. The following theorem is a Bonnet-Meyers type result which motivates our study of positively curved polytopes.

\begin{theorem}[Bonnet-Myers Theorem for Forman--Ricci curvature (see~\cite{forman})] \label{thm: forman diam}
    Let $P$ be a polytope. Suppose there exists $c > 0$ such that for  $\forman{e} \geq c$ for each edge $e\in E(P)$. Then the following hold:
        \begin{enumerate}[label=(\roman*)]
            \item If $v_1, v_2\in V(P)$ and $e_1, e_2\in E(P)$ occur on any shortest $v_1$-$v_2$ path, then the distance $d(v_1, v_2)$ satisfies
                \[
                d(v_1, v_2)\leq \frac{1}c\left(2 + |\mathcal{F}\uparrow(e_1)| + |\mathcal{F}\uparrow(e_2)| \right).
                \]
            \item Consequently,
                \[
                \mathrm{diam}(P) \leq \frac{2}{c}\left(1+\max_{e\in E(P)}|\mathcal{F}\uparrow(e)| \right).
                \]
        \end{enumerate}
\end{theorem}

Some of our results call only for a combinatorial graph which is not necessarily derived as the 1-skeleton of a polytope; in such cases we consider graphs of the form $G=(V, E)$ where $V$ is any finite set of vertices and $E\subseteq \binom{V}{2}$. If $\{i, j\}\in E$ we write $i\sim j$. If $G$ is any graph, we denote by $n\geq 1$ the number of vertices in $G$ and $m\geq 1$ the number of edges in $G$. We denote by $\mathbf{A}=\mathbf{A}(G)$ (resp. $\mathbf{D}=\mathbf{D}(G)$) the adjacency matrix (resp. diagonal vertex degree matrix) of $G$. The matrix $\mathbf{L}=\mathbf{D}-\mathbf{A}$ is known as the \emph{combinatorial Laplacian matrix} of $G$. We denote by $E'\subseteq V\times V$ any fixed but otherwise arbitrary orientation of the edges $E$, i.e., any set containing exactly one ordered representative for each $e=\{v_1, v_2\}\in E$. The \emph{vertex-edge oriented incidence matrix} $\mathbf{B}\in\mathbb{R}^{n\times m}$ is defined entrywise by the values
    \begin{align}\label{eq:defn-incidence}
        \mathbf{B}_{v_i, e_j} &= \begin{cases}
            1 &\text{ if }e_j = (v_i, \cdot)\\
            -1&\text{ if }e_j = (\cdot, v_i)\\
            0&\text{ otherwise}
        \end{cases},\quad v_i\in V, \hspace{.1cm}e_j\in E'.
    \end{align}
Note that regardless of choice of orientation on the edges, $\mathbf{L} = \mathbf{B}\mathbf{B}^\top$ (see, e.g.,~\cite{chung1997spectral}). We choose to orient the edges with respect to the indexing on the nodes only for concreteness. We recall the well known facts that L is symmetric and positive semidefinite, and as long as $G$ is connected, $\mathbf{L}$ has rank $n-1$. We denote the Moore-Penrose inverse of $\mathbf{L}$ by $\mathbf{L}^\dagger$ (see~\cite{moore1920reciprocal} for an historic reference). \emph{Effective resistance} is a metric on $V$ which is defined by the formula
    \begin{align}
        r_{v_1 v_2} = (\mathbf{1}_{v_1} - \mathbf{1}_{v_2})^\top \mathbf{L}^\dagger (\mathbf{1}_{v_1} - \mathbf{1}_{v_2}), \quad v_1, v_2\in V.
    \end{align}
Here, $\mathbf{1}_{v}$ is the indicator vector of $v\in V$. We note that by writing $\widetilde{\mathbf{L}} = \mathbf{L} + \frac{1}{n}\mathbf{J}_n$ (where $\mathbf{J}_n\in\mathbb{R}^{n\times n}$ is the all ones matrix), which is nonsingular, one may also write	
    \begin{align}
        r_{v_1 v_2} = (\mathbf{1}_{v_1} - \mathbf{1}_{v_2})^\top \widetilde{\mathbf{L}}^{-1} (\mathbf{1}_{v_1} - \mathbf{1}_{v_2}), \quad v_1, v_2\in V.
    \end{align}
The following variational characterization of effective resistance is useful in practice.

\begin{lemma}\label{prop:flow-resistance}
    For each $u, v\in V$, the effective resistance $r_{uv}$ is given by
        \begin{align*}
            r_{uv} &= \inf\left\{\|\mathbf{J}\|_2^2 \;:\;\mathbf{J}\in\mathbb{R}^{E'},\;  \mathbf{B}\mathbf{J} = \mathbf{1}_{u} - \mathbf{1}_{v}\right\}.
        \end{align*}
\end{lemma}

Its proof is straightforward linear algebra and is omitted.

\begin{definition}\label{defn:curvature-resistance}
    Let $G=(V, E)$ be any fixed graph and let $v\in V$ be fixed. Then the {\normalfont effective resistance curvature} at $v$, denoted $\resistance{v}$, is given by
        \begin{align*}
            \resistance{v} &= 1- \frac{1}{2}\sum_{\substack{u\in V\\ u\sim v}} r_{uv}.
        \end{align*}   
\end{definition}

This notion of curvature originally appeared in a 2022 paper of Devriendt and Lambiotte (see~\cite{DL}). A subsequent notion, also known as effective resistance curvature, was introduced in a 2024 paper of Devriendt, Ottolini, and Steinerberger (see~\cite{DOS}). The latter notion can be considered a modification of the former, as although it in principle is motivated by an equilibrium measure of the effective resistance matrix, the two are the same up to a global scaling factor. The latter paper obtained a Bonnet-Myers-type result, which we state below, having been adjusted to be consistent with our chosen convention~\cref{defn:curvature-resistance}.

\begin{theorem}[Bonnet-Myers Theorem for Resistance Curvature (see~\cite{DOS})]
    Let $G=(V,E)$ be a connected graph with maximum degree $\Delta$ and effective resistance matrix $\mathbf{R} = (r_{uv})_{u, v\in V}$. Assume the node resistance curvature $\boldsymbol{\kappa} = (\resistance{v})_{v\in V}$ satisfies $\resistance{v} \geq K > 0$ for each $v\in V$. Then
        \begin{align*}
            \diam(G)\leq \left\lceil \sqrt{\frac{\Delta\; \boldsymbol{\kappa}^\top\mathbf{R}\boldsymbol{\kappa}}{K}}\;\log|V|  \right\rceil.
        \end{align*}
\end{theorem}

\subsection{Our Contributions}

We study various aspects of the Forman--Ricci and effective resistance curvatures for skeleta of polytopes. We start by analyzing the average curvature and derive an equation to compute average curvature in terms of face numbers of the polytope. By analyzing polytopes whose $2$-skeletons admit an edge-transitive group action we obtain the following result. We say a polytope $P$ is \emph{Forman--Ricci-positive} provided $\forman{e} > 0$ for each $e\in E(P)$. We remind the reader that we consider polytopes up to combinatorial equivalence. 

\begin{theorem}\label{thm:FormanInfinite6}
    For each $d\geq 6$, there are infinitely many Forman--Ricci-positive polytopes with dimension $d$.
\end{theorem}

Further analysis of the average curvature yields the following result which is useful for studying positive polytopes in smaller dimensions. 

\begin{theorem}\label{thm: boundedDeg}
    Let $d\geq 3$ be fixed and let $\Delta\ge 3$ be a real number. The set of Forman--Ricci-positive $d$-polytopes with the property that the average degree of a vertex is at most $\Delta$ is finite. 
\end{theorem}

This theorem has several consequences and essentially says that the edge density of Forman--Ricci-positive graphs has to be rather large. As a consequence, the number of Forman--Ricci-positive \emph{simple} $d$-dimensional polytopes is finite for all $d$. 

Next we turn to the situation in low dimensions.

\begin{theorem}\label{thm:FormanFinite3}
    The set of Forman--Ricci positive $3$-polytopes is finite. Polytopes in this collection have no more than $15$-vertices. 
\end{theorem}

We illustrate the graphs of each of the Forman--Ricci positive 3-polytopes in~\cref{fig:forman_positives}. It is interesting to compare~\Cref{thm:FormanFinite3} with a similar result on a different combinatorial curvature in~\cite{DevosMohar2007}.

Next, in dimension $4$, using the proof of~\Cref{thm: boundedDeg} and known structural results about the graphs of $3$-polytopes, we obtain the following result.

\begin{theorem}\label{thm:FormanFinite4}
    The set of Forman--Ricci positive $4$-polytopes is finite.
\end{theorem}

The result says little about how to classify such $4$-polytopes, but the proof shows that, in particular, Forman--Ricci positive $4$-polytopes have no vertex of degree greater than 12. The case of $d=5$ is less well understood, and we conjecture that~\cref{thm:FormanFinite3} and~\cref{thm:FormanFinite4} extend to this setting.

\begin{conjecture}\label{conjecture:5-polytopes}
    The set of Forman--Ricci positive $5$-polytopes is finite.
\end{conjecture}

In dimension five, we are able make progress in the special case of simplicial polytopes: Recall that a $d$-dimensional polytope is said to be \emph{simplicial} if all of its faces (excluding $P$ itself) are simplices.

\begin{theorem}
    The set of Forman--Ricci positive simplicial $5$-polytopes is finite.
\end{theorem}

This appears to be strong evidence in favor of~\cref{conjecture:5-polytopes}, since $2$-faces that are not triangles contribute at most $0$ to the curvature computation.

Next, we describe our results on resistance curvature. We follow a similar program with a more quantitative angle and obtain several results which, although spiritually analogous, have different conclusions and implications. We call a polytope $P$ \emph{resistance positive} if $\resistance{v} > 0$ for each $v\in V(P)$.

\begin{figure}[t!]
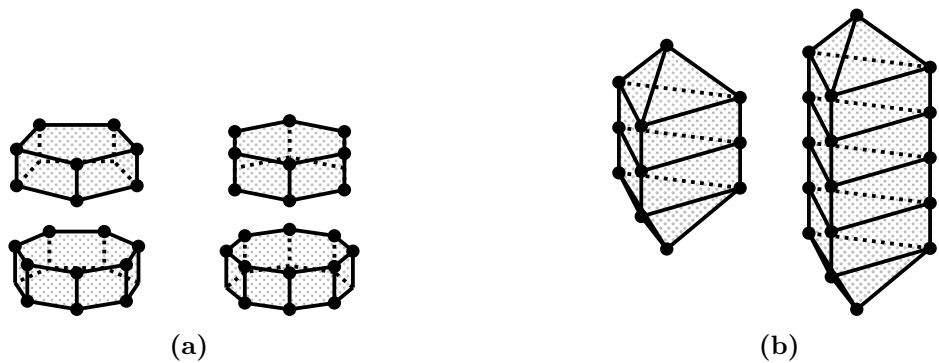

    \centering
    \begin{subfigure}[t]{0.48\textwidth}
        \centering
        \input{tikz/prism}
        \caption{}\label{subfig:resis-prism}
    \end{subfigure}
    \begin{subfigure}[t]{0.48\textwidth}
        \centering
        \input{tikz/pencil}
        \caption{}\label{subfig:resis-pencil}
    \end{subfigure}%
    ~
    \caption{\emph{(a)} The prism $3$-polytopes and their graphs with faces consisting of (in clockwise order) five, six, seven, and eight vertices. \emph{(b)} The $3$-polytope constructed via its graph $G_k$ in~\cref{subsec:resistance-positive-polyhedra}, here shown for $k=3, 5$.}\label{fig:resistance-curvature}
\end{figure}

\begin{theorem}\label{thm:resistance-positive-3polytopes}
    For each $d\geq 2$, there are infinitely many resistance positive polytopes with dimension $d$.
\end{theorem}

\Cref{thm:resistance-positive-3polytopes} follows from the existence of an infinite family of $d$-polytopes with vertex transitive graphs; namely, in the case of $d=2,3$, polygons and polygonal prisms (see~\cref{subfig:resis-prism}); and in the case of $d\geq 4$, the existence of $d$-polytopes whose graphs are isomorphic to the complete graph $K_n$. We note, however, that a complete characterization of resistance positive $3$-polytopes seems out of reach at present; in particular, we identify a family of resistance positive $3$-polytopes which have graphs that are not vertex transitive (see~\Cref{subsec:resistance-positive-polyhedra} and~\Cref{subfig:resis-pencil}). 

On the quantitative side we obtain the following curvature bound for a vertex in a simple polytope in terms of the lengths of the polygonal cycles incident to a given vertex.

\begin{theorem}\label{thm:curvature-bound-intro}
    Let $G=(V,E)$ be the $1$-skeleton of a simple $3$-polytope. Fix $v\in V$, and let $\mathcal{C
    }(v)$ denote the set of $2$-faces incident to $v$. For each $C\in\mathcal{C}(v)$, let $\ell_C := |E(C)|$ be the length (edge count) of $C$. Then the resistance curvature of $G$ at $v$ satisfies
        \begin{align*}
            \resistance{v} \geq
            1 \;-\; \frac{1}{2}\sum_{C\in\mathcal{C}(v)}
            \frac{\ell_C-1}{
                (\ell_C-1)\!\left(\sum_{C'\in\mathcal{C}(v)}\frac{1}{\ell_{C'}-1}+1\right)-1}.
        \end{align*}
\end{theorem}

This bound is used to identify families of non-vertex-transitive $3$-polytopes obtained as $\Delta$-expansions of known simple $3$-polytopes (see~\Cref{thm:delta-expansions} and examples in~\Cref{fig:delta-expansions}). We also investigate quantitative lower bounds for the resistance curvature in generic graphs, and obtain the following degree-based criterion for the existence of a vertex with negative resistance curvature.

\begin{corollary}\label{cor:intro-resistance-curv-lowerbound}
    Let $G=(V,E)$ be any graph, and suppose $v\in V$ satisfies the following two conditions:
        \begin{enumerate}[label={\normalfont (\roman*)}]
            \item $d_v\geq 2$, and
            \item For each $u\sim v$, $d_u\leq d_v -2$.
        \end{enumerate}
    Then the resistance curvature $\resistance{v}$ satisfies $\resistance{v}\leq 0$.
\end{corollary}

\Cref{cor:intro-resistance-curv-lowerbound} can be used to rule out resistance positivity for many polytopes; pyramids are natural examples in the case of $d=3$ since their apexes generally meet the hypotheses of~\Cref{cor:intro-resistance-curv-lowerbound}. Moreover,~\Cref{cor:intro-resistance-curv-lowerbound} establishes that resistance positive graphs must, in a weak sense, be ``close'' to degree regular, and in doing so lends credence to the overall picture that resistance positive polytopes are often rare.

\section{Forman--Ricci curvature of polytopes}\label{sec: forman results}
In this section we consider the case of Forman--Ricci curvature of the 2-dimensional skeleta of polytopes. In~\Cref{subsec:avg-curvature} we record general facts valid for all polytopes and which are useful across dimensions. We compute Forman--Ricci curvature for simplices and hypercubes, from which~\Cref{thm:FormanInfinite6} follows. In~\Cref{subsec:avg-degree} we study polytopes whose graphs have bounded degree and show that, in any fixed dimension, there are only finitely many simple positive polytopes. Next, in~\Cref{subsec:forman-dim3,subsec:forman-dim4} we prove~\Cref{thm:3dFormanFinite,thm:4-dimFiniteness}, establishing finiteness in dimensions $3$ and $4$. Finally, in~\Cref{subsec:simplicial-five} we examine $d=5$; while the picture remains open, our partial results on simplicial $5$-polytopes point toward finiteness.

\subsection{Average curvature, symmetry, and high dimensions}\label{subsec:avg-curvature}
The goal of this subsection is to compute the Forman--Ricci curvature of the $2$-skeleton of the $d$-simplex and the hypercube. Since the automorphism groups of these polytopes act transitively on their $2$-skeleta, the curvature is constant on each edge.

Therefore by computing the average curvature of an edge in a polytope and specializing to the two above cases, we may recover their Forman--Ricci curvature. In order to do this, we will show that the average curvature across all edges depends on what are known as the \emph{flag}, or \emph{$f$-numbers}, of the polytope. 

The average curvature of a $d$-dimensional polytope $P$ is defined as
    \begin{align*}
        \mathcal{K}(P):= \frac{1}{f_1(P)} \sum_{e} \forman{e}.
    \end{align*}
For $k=0, 1$, we let $f_{k2}= f_{k2}(P)$ to be the numbers of pairs $(x,F)$ where $x$ is a $k$-face of $P$, and $F$ a two dimensional face that contains $x$. Furthermore, $d_k(P)$ denotes the number of vertices of degree $k$ of $P$ and $p_k(P)$ denotes the number of $2$-faces that are $k$-gons.

\begin{lemma}\label{lemma:AverageCurvatureIdentity}
    Let $P$ be a $d$-polytope. The following equation holds: 
        \begin{align*}
            \mathcal{K}(P)= \frac{1}{f_1(P)}\left(6f_{02}(P) + 4f_1(P) - \sum_{k\ge d} k^2d_k(P) - \sum_{k\ge 3}k^2p_k(P)\right).
        \end{align*}
\end{lemma}

\begin{proof}
    Let $E_{||}(P)$ be the set of pairs $(e,e')$ of edges that are parallel. Furthermore, along the lines of~\cref{def:parallel-neighbors}, let $E\uparrow(P)$ be the set of ordered pairs of disjoint edges that are in a common $2$-face, and let $E\downarrow(P)$ be the set of ordered pairs of edges that have a common vertex. Note that $E\uparrow(P)$ and $E\downarrow(P)$ contain edges which are not parallel. Notice that $|E\uparrow(P)| = \sum_{k\ge 3} k(k-1)p_k$ and $|E\downarrow(P)| = \sum_{k\ge d} k(k-1)d_k$. Parallel edges correspond to pairs that are either in $E\uparrow(P)$ or in $E\downarrow(P)$, but not both. The pairs of edges appearing in both places are exactly the pairs contained in a single $2$-face and which share an endpoint. It follows that $|E_{||}(P)| = |E\uparrow(P)| + |E\downarrow(P)| - 4f_{02}$. It follows that:
        \begin{eqnarray*}
            \mathcal{K}(P)f_1 &=& \sum_{e\in E(P) } \left(F(e) + 2 -\mathcal{F}(P,e)\right) \\
            &=&  f_{12}+ 2f_1 - |E_{||}(P)| \\
            &=& f_{12} + 2f_1 - (|E\uparrow(P)| + |E\downarrow(P)| - 4f_{02}) \\ 
            &=& 5f_{02} + 2f_1 - \sum_{k\ge d} k(k-1)d_k - \sum_{k\ge 3}k(k-1)p_k \\ 
            &=& 6f_{02} + 4f_1 - \sum_{k\ge d} k^2d_k - \sum_{k\ge 3}k^2p_k.
        \end{eqnarray*}
     Where we also us that $f_{02}=f_{12}=\sum_{k\ge 3} kp_k$.
\end{proof}

\begin{corollary}
    Let $d\ge 3$. The Forman--Ricci curvature of any edge of the $d$-dimensional simplex is $d+1$ and the Forman--Ricci curvature of any edge of the $d$-dimensional hypercube is equal to $2$.
\end{corollary}

\begin{proof}
    Notice that in both cases the automorphism group of the two skeleton of a complex acts transitively on the edges, which implies that each of the edge curvature values are equal and their common value is realized by the average curvature. To compute the curvature of simplex, we have that $f_0= n+1$, $f_1= \binom{n+1}{2}$, $f_{02}= 3f_2= 3\binom{n+1}3$ $d_n= f_0$, $d_k=0$ for $k>n$, $p_3=f_2 = \binom{n+1}{3}$ and $p_k=0$ for $k>3$. Plugging this into~\cref{lemma:AverageCurvatureIdentity} yields the result. To compute the curvature of hypercube, we have that $f_0= 2^n$, $f_1= n2^{n-1}$, $f_{02}= 4f_2= \binom{n}{2}2^{n}$ $d_n= f_0$, $d_k=0$ for $k>n$, $p_4=f_2 = \binom{n}{2}2^{n-2}$ and $p_k=0$ for $k\not=4$. Plugging this into ~\cref{lemma:AverageCurvatureIdentity} yields the result. 
\end{proof}

\begin{corollary}
    Let $d\ge 6$ be an integer. There are infinitely many positive Forman--Ricci polytopes of dimension $d$.
\end{corollary}

\begin{proof}
    Fix an integer $d\ge 6$. For every $n\ge 6$ there exist both \emph{(i)} a $2$-neighborly $d$-polytope on $n$ vertices and \emph{(ii)} a $2$-neighborly cubical $d$-polytope on $n$ vertices. Let $P$ be either such polytope. Then the $2$-dimensional skeleton of $P$ coincides with that of, respectively, a simplex or a cube on the same vertex set; consequently, the Forman--Ricci curvature of each edge in $P$ agrees with that of the corresponding simplex or cube and is, in particular, everywhere-positive.
\end{proof}

\subsection{Forman--Ricci-positive polytopes and the average degree of a vertex}\label{subsec:avg-degree}

In this subsection, we show that the number of $d$-polytopes whose average vertex degree is bounded above by a constant is finite. The idea is to give a bound on the number of vertices. We recall the following lemma known as the Moore bound:

\begin{lemma}[Moore Bound~\cite{hoffman1960moore,miller2012moore}]\label{lem: degdiam}
   Let $G = (V,E)$ be a graph with maximum vertex degree $\Delta \neq 2$ and diameter $D$. Then, the number of vertices $|V|$ is bounded by
         \begin{align*}
             |V| \leq \frac{\Delta (\Delta - 1)^D - 2}{\Delta - 2}.
         \end{align*}
\end{lemma} 

\begin{theorem}\label{thm:average-degree}
     Let $P$ be a $d$-dimensional Forman--Ricci-positive polytope and let $\avgdegree$ denote the average vertex degree of $P$. Then the number of vertices $n$ of $P$ satisfies
        \begin{align*}
            n\leq \frac{2^{3\avgdegree + 1}\avgdegree( 2^{3\avgdegree + 1}\avgdegree-1)^{4+6\avgdegree}-2}{ 2^{3\avgdegree }\avgdegree-1}.
        \end{align*}
\end{theorem}

\begin{proof}
    Let $v$ be a vertex of $P$ achieving the minimum vertex degree of the polytope and let $G$ be the subgraph of the graph of $P$ induced by the set of vertices
        \begin{align*}
            S &= \{w\in V(P) \;:\; d(w,v)\leq (1 + \deg(v) + 2\avgdegree)\}.
        \end{align*}
    If $w$ is a vertex of an edge $e$, then $|\mathcal{F}\uparrow(e)|\le \deg(w)-1$. Since the Forman--Ricci curvature of $P$ satisfies $|\forman{\cdot}|\geq 1$ it follows from~\cref{thm: forman diam} that each vertex $w$ of degree at most $2\avgdegree$ satisfies $w\in S$. Furthermore, the diameter of $G$ is at most $2(1 + \deg(v)+ 2\avgdegree)$, using the paths passing through $v$. 

    Next we bound the maximum degree $\Delta$ of $G$. Fix an edge $e=\{u, v\}\in E(G)$. If $\deg(u) >2\deg(v)$ holds, we must have that $|\mathcal{F}\uparrow(e)| \le \deg(v)-1$. Among the $\deg(u) - 1$ edges incident to $u$ excluding $e$, at most $|\mathcal{F}\uparrow(e)|$ share a common $2$-face with $e$. Hence at most
        \begin{align*}
            (\deg(u)-1)-|\mathcal{F}\uparrow(e)|\ \ge\ \deg(u)-\deg(v)
        \end{align*}
    of them are parallel neighbors of $e$ satisfying the condition \emph{(i)} in~\cref{def:parallel-neighbors}. Therefore
        \begin{align*}
            \forman{e} &= |\mathcal{F}\uparrow(e)|+2-|\mathcal{E}(e)| \\
            &\leq(\deg(v)-1)+2-(\deg(u)-\deg(v)) \\
            &=\ 1+2\deg(v)-\deg(u)\ \le\ 0,
        \end{align*}
    a contradiction. Consequently, we must have 
        \begin{align*}
            \max\{\deg(u),\deg(v)\}\le 2\,\min\{\deg(u),\deg(v)\}
        \end{align*}
    for each edge $e\in E(G)$. By induction along a path, any vertex $w$ at distance $\Delta$ from $v$ satisfies $\deg(w)\le 2^\Delta\deg(v)$. Since we consider vertices within distance $1+\deg(v)+2\avgdegree$ of $v$, each vertex of $G$ has degree at most
        \begin{align*}
            2^{\,1+\deg(v)+2\avgdegree}\deg(v)\ \le\ 2^{\,3\avgdegree+1}\avgdegree,
        \end{align*}
    because $\deg(v)\le \avgdegree$ by the choice of $v$ as a minimum-degree vertex. As a consequence, by the Moore Bound (\cref{lem: degdiam}), 
        \begin{eqnarray*}
            |V(G)| &<& \frac{2^{3\avgdegree + 1}\avgdegree( 2^{3\avgdegree + 1}\avgdegree-1)^{2+2\deg(v) +4\avgdegree }-2}{ 2^{3\avgdegree + 1}\avgdegree-2} \\  &\le& \frac{2^{3\avgdegree + 1}\avgdegree( 2^{3\avgdegree + 1}\avgdegree-1)^{2+6\avgdegree }-2}{ 2^{3\avgdegree + 1}\avgdegree-2}. 
        \end{eqnarray*}
    Lastly we argue that the number of vertices $n$ of $P$ is no more than twice the number of vertices of $G$. Partition the vertex of $P$ into sets
        \begin{align*}
            A_1 &= \{v\in V(P) \;:\; \deg(v) \leq \avgdegree\},\\
            A_2 &= \{v\in V(P) \;:\; \rho < \deg(v) < 2\avgdegree\},\\
            A_3 &= \{v\in V(P) \;:\; 2\rho \leq \deg(v)\},
        \end{align*}
    and let $a_i = |A_i|$ for $i=1,2,3$. Then $n=a_1+a_2+a_3$ and since all vertices of degree at most $2\avgdegree$ belong to $G$, we have that $|V(G)| \ge a_1+a_2$. Since each vertex of $P$ has degree at least $d$, it follows that
        \begin{align*}
            \avgdegree \ge \frac{a_1d +a_2\avgdegree +2a_3\avgdegree}{a_1+a_2+a_3}.
        \end{align*}
    This implies that $a_1\ge \frac{\avgdegree}{\avgdegree - d}a_3 \ge a_3$. So $2|V(G)| \ge 2(a_1+a_2) \ge n+a_2 \ge n$. The claim follows.
\end{proof}

The upper bounds on the number of vertices in the above result are far from tight. Nevertheless, we can exploit~\cref{thm:average-degree} to obtain two corollaries.

\begin{corollary}\label{cor: finiteAverage}
    Let $d\geq 3$ be fixed and let $\Delta\ge 3$ be a real number. The set of Forman--Ricci positive $d$-polytopes such that the average degree of a vertex is at most $\Delta$ is finite. 
\end{corollary}

\begin{corollary}\label{cor: finite degree}
   Let $d\ge 3$ and $\Delta\ge d$ be fixed positive integers. There are finitely many Forman--Ricci-positive $d$-polytopes whose maximum degree is at most $\Delta$. In particular, there are finitely many Forman--Ricci-positive \textit{simple} $d$-polytopes.
\end{corollary}

\subsection{Forman--Ricci curvature in Dimension 3}\label{subsec:forman-dim3}

In this subsection, we discuss in more detail the case of $3$-dimensional polytopes. First we show that the number of such polytopes is finite.

\begin{theorem}\label{thm:3dFormanFinite}
    There are finitely many Forman--Ricci-positive $3$-polytopes.
\end{theorem}

\begin{proof}
    We know from Steinitz's theorem that if $P$ is a $3$-polytope, then $f_1(P) \le 3f_0(P)-6$. Since the average vertex degree $\rho$ of $P$ is $\frac{2f_1}{f_0}$, we have that $\rho\leq 6$. The claim then follows from~\cref{thm:average-degree}. 
\end{proof}

Since combinatorial types of $3$-polytopes correspond to planar $3$-connected graphs, we are able to say much more about them. In fact, since each edge is contained in exactly two facets, the curvature $\forman{e}$ must satisfy $\forman{e}\leq 4$. This leads to a classification of all Forman--Ricci-positive $3$-polytopes; to get there we first collect a few structural results and some additional considerations.

Recall that the \emph{polar} $P^\ast$ of a polytope $P$, realized as a set $P\subseteq\mathbb{R}^{d}$, is the set
    \begin{align}
        P^\ast &= \{\mathbf{y}\in\mathbb{R}^d \;:\; \mathbf{y}^\top \mathbf x\leq 1\text{ for each }\mathbf{x}\in P\}.
    \end{align}
More generally, if $F$ is a face of $P$, we denote by $F^\ast$ its polar. We make the following simple observation.  The following lemma is well known and a proof is omitted (see~\cite[Sec. 3.4]{Grunbaum}).

\begin{lemma}\label{lem:duality}
    Let $P$ be a $d$-dimensional polytope and let $P^*$ be (any realization of) the polar of $P$. For any $k$-face $F$ of $P$, let $F^*$ be be corresponding $d-1-k$-dimensional dual face. Then $\mathcal{F}_k(F)= \mathcal{F}_{d-1-k}(F^*)$.
\end{lemma}

\begin{lemma}\label{lem:simplicial-positive-dual}
    Let $P$ be a fixed $3$-polytope. Then the correspondence between the edges of $P$ and the edges of $P^*$ preserves Forman--Ricci curvature. In particular, $P$ is Forman--Ricci positive if and only if $P^*$ is Forman--Ricci positive. 
\end{lemma}

\begin{proof}
    For $3$-polytopes the dual of an edge is an edge and parallel edges of the two different types are swapped by this operation; in particular, a vertex of degree $k$ corresponds to a $k$-gon in the dual.
\end{proof}

We now study some rigidity results for Forman--Ricci-Positive polytopes. 

\begin{lemma} \label{lem: 6 bounds}
    Let $P$ be a Forman--Ricci-positive $3$-polytope. Then the maximum degree $\Delta$ of a vertex and maximum number of sides of a $2$-face are both at most $6$. Furthermore, if $P$ has a vertex of degree $6$ or a hexagonal face, then it is a hexagonal pyramid.
\end{lemma}

\begin{proof}
    For $k\geq 3$, each $k$-gon in the $2$-skeleton of a $3$-polytope is dual to a vertex of degree $k$, hence by Lemma~\cref{lem:simplicial-positive-dual}, it suffices to consider the case of $k$-gons. If $F$ is a $k$-gon with $k\ge 6$, then any edge $e$ in the $k$-gon has $k-3$ parallel edges. Since the $e$ contains two vertices and is contained in two facets, it follows that $\mathcal{F}(e)\le 4-(k-3)= 7-k$. Therefore, $e$ is not Forman--Ricci positive if $k\ge 7$. 
    
    Furthermore, if $k=6$, then the curvature on any edge is automatically at most one. To avoid the addition of parallel edges, every vertex of the hexagon must have degree $3$, and all the polygons adjacent to the edges have to be triangles, which means that $P$ is a pyramid over the hexagonal face. 
\end{proof}

We remark that the hexagonal pyramid is combinatorially self dual and the Forman--Ricci curvature is constant and equal to one on every edge.

\begin{proposition}\label{obs: diam = 6}
    Let $P$ be a $3$-polytope with everywhere-positive Forman--Ricci curvature. Then $\diam(P) \leq 6$.
\end{proposition}

~\cref{obs: diam = 6} follows immediately from~\cref{thm: forman diam} and the fact that In a $3$-polytope, each edge $e$ satisfies $\max_{e\in E(P)}|\mathcal{F}\uparrow(e)|=2$.

We are now classify simple $3$-polytopes. By Lemma~\cref{lem:simplicial-positive-dual}, this also classifies positive simplicial $3$-polytopes. 

\begin{theorem}\label{thm: simple 3 polys.}
    There are exactly five simple Forman--Ricci-positive  $3$-polytopes.  
\end{theorem}

\begin{proof}
    Assume that $P$ is a Forman--Ricci-Positive simple $3$-polytope. We observe first that $3p_3+2p_4+p_5=12$ and that no two pentagons are edge adjacent. Let the number of pentagons be denoted $k\geq 0$. We know that any pair of pentagons is disjoint: a common edge is necessarily negative and a common vertex would have degree at least $4$. Moreover a quadrilateral shares and edge with at most $2$ pentagons and a triangle must be incident to at most one square. Since each edge must be incident to two polygons we get that $5k \le p_3 + 2p_4= 12-2p_3-k$, or equivalently, $k\le 2 - \frac{p_3}3$. So the possible values of $k$ are $0,1,2$, and we can proceed in cases.
    
    \begin{enumerate}[label=\emph{(\roman*)}]

        \item If $k=0$, then $3p_3+2p_4= 12$. The solutions to this equation are $(p_3,p_4) \in \{(4,0), (2,3), (0,6)\}$. By inspection, these can only be realized by a tetrahedron, a triangular prism, or a cube, respectively. 
    
        \item If $k=1$, then $3p_3+2p_4= 11$. The solutions to this equation are $(p_3,p_4) \in \{(3,1), (1,4)\}$. The first case does not have enough facets so that each edge of the pentagon is contained in $2$ polygonal faces. The second case would yield a simple $3$-polytope with $6$ faces, $12$ edges and $8$ vertices, meaning that the vertices \emph{not} incident to the pentagon have exactly $2$ edges between them (there are 5 edges in the pentagon, and 5 edges out each vertex of the pentagon). One of the vertices not in the pentagon is connected to the other two and so it has exactly one edge connecting it to the vertices of the pentagon. The remaining $4$ vertices of the pentagon must be connected to the remaining two points not in the pentagon, in such a way that each non-pentagon vertex is connected to two pentagon vertices. There are two ways to do this and none of them produces a desired polytope.

        \item If $k=2$, then $p_3=0$ and $P$ is a prism over a pentagon.
    \end{enumerate} 

\end{proof}

We now prove a theorem that allows us to extend the classification beyond the simple and simplicial cases. The proof as sketched reduces to a thorough case-by-case analysis implemented by checking a database of planar $3$-connected graphs with few vertices, which can be easily implemented. 

\begin{theorem}\label{thm:3polytope-degree}
    If $P$ is a Forman--Ricci-positive $3$-dimensional polytope, then $f_0(P) \le 16$ or $f_2(P)\le 16$. 
\end{theorem}
    
\begin{proof}
    Unless $P$ is a pryamid over a hexagon, all vertices have degree at most five and all $2$-dimensional faces have at most five sides. We separate the proof in two different cases: first assuming $P$ has a pentagonal face or a vertex of degree five, and second assuming otherwise.

    To this end, assume $P$ contains a pentagon or a vertex of degree five. We assume there is a pentagon, and the case of a degree five vertex follows by duality. If $F$ is a pentagonal face of $P$ and $e$ is a edge of $F$, then $\mathcal{F}(e) \le 2$, so the value of the curvature is $2$ if the degree of the vertices is $3$ and the other incident facet is a triangle. It can be equal to one if it has one vertex of degree three and one of degree four, and an adjacent triangle, or two edges of degree three and an adjacent quadrilateral. In particular, the degrees of all the vertices in the pentagonal facet are three or four and no pair of adjacent vertices have degree four. 

    Thus there are 3 cases to consider for the degrees of vertices in the pentagon. They can all be handed similarly, so we will explain one of them in detail and the rest follow. In the case when there is exactly one vertex of degree $4$ in the pentagonal facet, then the two edges of the pentagon incident to this edge are then contained in triangles and $G(P)$ contains an induced subgraph isomorphic the  following, drawn as a Schlegel diagram with the pentagon as its boundary as seen in~\cref{fig:pentagon1}.

    \begin{figure}
        \begin{subfigure}{0.3\textwidth}
            \centering
            % Requires: \usepackage{tikz} \usetikzlibrary{patterns}
\begin{tikzpicture}[scale=0.5]
  % --- Tunable geometry (kept simple to match your style) ---
  \def\xTop{2.6}   % half-length of the top edge of the base polygon
  \def\yTop{0.9}   % y of the top edge
  \def\xBot{1.8}   % half-length of the bottom edge
  \def\yBot{-2.0}  % y of the bottom edge
  \def\xGap{0.95}  % half-gap between the two roof wedges along the top edge
  \def\hA{3.1}     % apex height

  % --- Base polygon (trapezoid-like pentagon outline as shown) ---
  \coordinate (L)  at (-\xTop,\yTop);   % top-left
  \coordinate (R)  at ( \xTop,\yTop);   % top-right
  \coordinate (BR) at ( \xBot,\yBot);   % bottom-right
  \coordinate (BL) at (-\xBot,\yBot);   % bottom-left

  % --- Roof geometry ---
  \coordinate (A)  at (0,\hA);          % apex
  \coordinate (ML) at (-\xGap,\yTop);   % inner left point on top edge
  \coordinate (MR) at ( \xGap,\yTop);   % inner right point on top edge

  % --- Shaded wedges (crosshatch to match reference figures) ---
  \path[pattern=crosshatch dots, pattern color=gray!55] (A) -- (L) -- (ML) -- cycle;
  \path[pattern=crosshatch dots, pattern color=gray!55] (A) -- (R) -- (MR) -- cycle;

  % --- Outline: thick black edges like in your snippets ---
  \draw[line width=0.5mm, black] (R) -- (BR) -- (BL) -- (L);  % base polygon
  \draw[line width=0.5mm, black] (A) -- (L);
  \draw[line width=0.5mm, black] (A) -- (R);
  \draw[line width=0.5mm, black] (A) -- (ML);
  \draw[line width=0.5mm, black] (A) -- (MR);
  \draw[line width=0.5mm, black] (L) -- (ML); % top edge drawn last for clean bases
  \draw[line width=0.5mm, black] (R) -- (MR); % top edge drawn last for clean bases

  % --- Edge labels ---
  \node at (-3.0,-0.3) {$e_{1}$};
  \node at (0,-2.55) {$e_{2}$};
  \node at ( 3.0,-0.3) {$e_{3}$};
\end{tikzpicture}
            \caption{The (partial) Schlegel diagram}
            \label{fig:pentagon1}
        \end{subfigure}\hfill
        \begin{subfigure}{0.65\textwidth}
            \centering
            % Requires: \usepackage{tikz} \usetikzlibrary{patterns}
\begin{tikzpicture}[scale=0.35]
  % --- Tunable geometry (kept simple to match your style) ---
  \def\xTop{2.6}   % half-length of the top edge of the base polygon
  \def\yTop{0.9}   % y of the top edge
  \def\xBot{1.8}   % half-length of the bottom edge
  \def\yBot{-2.0}  % y of the bottom edge
  \def\xGap{0.95}  % half-gap between the two roof wedges along the top edge
  \def\hA{3.1}     % apex height

  % --- Base polygon (trapezoid-like pentagon outline as shown) ---
  \coordinate (L)  at (-\xTop,\yTop);   % top-left
  \coordinate (R)  at ( \xTop,\yTop);   % top-right
  \coordinate (BR) at ( \xBot,\yBot);   % bottom-right
  \coordinate (BL) at (-\xBot,\yBot);   % bottom-left

  % --- Roof geometry ---
  \coordinate (A)  at (0,\hA);          % apex
  \coordinate (ML) at (-\xGap,\yTop);   % inner left point on top edge
  \coordinate (MR) at ( \xGap,\yTop);   % inner right point on top edge

  % --- Shaded wedges (crosshatch to match reference figures) ---
  \path[pattern=crosshatch dots, pattern color=gray!55] (A) -- (L) -- (ML) -- cycle;
  \path[pattern=crosshatch dots, pattern color=gray!55] (A) -- (R) -- (MR) -- cycle;
  \path[pattern=crosshatch dots, pattern color=gray!55] (BL) -- (ML) -- (L) -- cycle;
  \path[pattern=crosshatch dots, pattern color=gray!55] (BL) -- (BR) -- (ML) -- cycle;

  \path[pattern=crosshatch dots, pattern color=gray!55] (BR) -- (ML) -- (MR) -- (R) -- cycle;

  % --- Outline: thick black edges like in your snippets ---
  \draw[line width=0.5mm, black] (R) -- (BR) -- (BL) -- (L);  % base polygon
  \draw[line width=0.5mm, black] (A) -- (L);
  \draw[line width=0.5mm, black] (A) -- (R);
  \draw[line width=0.5mm, black] (A) -- (ML);
  \draw[line width=0.5mm, black] (A) -- (MR);
  \draw[line width=0.5mm, black] (L) -- (ML); % top edge drawn last for clean bases
  \draw[line width=0.5mm, black] (R) -- (MR); % top edge drawn last for clean bases

  \draw[line width=0.5mm, black] (MR) -- (ML);
  \draw[line width=0.5mm, black] (BR) -- (ML);
  \draw[line width=0.5mm, black] (BL) -- (ML);

  % --- Edge labels ---
  \node at (-3.0,-0.3) {$e_{1}$};
  \node at (0,-2.55) {$e_{2}$};
  \node at ( 3.0,-0.3) {$e_{3}$};

  % shapes of nodes
  \node[fill=black, minimum size=2mm] at (A) {};
  \node[fill=black, minimum size=2mm] at (L) {};
  \node[fill=black, minimum size=2mm] at (R) {};
  \node[fill=black, minimum size=2mm] at (BL) {};
  \node[fill=black, minimum size=2mm] at (BR) {};
  \node[fill=black, minimum size=2mm] at (ML) {};
  \node[circle, fill=black, minimum size=1mm] at (MR) {};
\end{tikzpicture}
            % Requires: \usepackage{tikz} \usetikzlibrary{patterns}
\begin{tikzpicture}[scale=0.35]
  % --- Tunable geometry (kept simple to match your style) ---
  \def\xTop{2.6}   % half-length of the top edge of the base polygon
  \def\yTop{0.9}   % y of the top edge
  \def\xBot{1.8}   % half-length of the bottom edge
  \def\yBot{-2.0}  % y of the bottom edge
  \def\xGap{0.95}  % half-gap between the two roof wedges along the top edge
  \def\hA{3.1}     % apex height

  % --- Base polygon (trapezoid-like pentagon outline as shown) ---
  \coordinate (L)  at (-\xTop,\yTop);   % top-left
  \coordinate (R)  at ( \xTop,\yTop);   % top-right
  \coordinate (BR) at ( \xBot,\yBot);   % bottom-right
  \coordinate (BL) at (-\xBot,\yBot);   % bottom-left

  % --- Roof geometry ---
  \coordinate (A)  at (0,\hA);          % apex
  \coordinate (ML) at (-\xGap,\yTop);   % inner left point on top edge
  \coordinate (MR) at ( \xGap,\yTop);   % inner right point on top edge

  % --- Shaded wedges (crosshatch to match reference figures) ---
  \path[pattern=crosshatch dots, pattern color=gray!55] (A) -- (L) -- (ML) -- cycle;
  \path[pattern=crosshatch dots, pattern color=gray!55] (A) -- (R) -- (MR) -- cycle;
  \path[pattern=crosshatch dots, pattern color=gray!55] (BL) -- (ML) -- (L) -- cycle;
  \path[pattern=crosshatch dots, pattern color=gray!55] (BL) -- (BR) -- (ML) -- cycle;

  \path[pattern=crosshatch dots, pattern color=gray!55] (BR) -- (ML) -- (MR) -- (R) -- cycle;

  % --- Outline: thick black edges like in your snippets ---
  \draw[line width=0.5mm, black] (R) -- (BR) -- (BL) -- (L);  % base polygon
  \draw[line width=0.5mm, black] (A) -- (L);
  \draw[line width=0.5mm, black] (A) -- (R);
  \draw[line width=0.5mm, black] (A) -- (ML);
  \draw[line width=0.5mm, black] (A) -- (MR);
  \draw[line width=0.5mm, black] (L) -- (ML); % top edge drawn last for clean bases
  \draw[line width=0.5mm, black] (R) -- (MR); % top edge drawn last for clean bases

  \draw[line width=0.5mm, black] (MR) -- (ML);
  \draw[line width=0.5mm, black] (BR) -- (MR);
  \draw[line width=0.5mm, black] (BL) -- (ML);

  % --- Edge labels ---
  \node at (-3.0,-0.3) {$e_{1}$};
  \node at (0,-2.55) {$e_{2}$};
  \node at ( 3.0,-0.3) {$e_{3}$};

    % shapes of nodes
  \node[fill=black, minimum size=2mm] at (A) {};
  \node[fill=black, minimum size=2mm] at (L) {};
  \node[fill=black, minimum size=2mm] at (R) {};
  \node[fill=black, minimum size=2mm] at (BL) {};
  \node[fill=black, minimum size=2mm] at (BR) {};
  \node[fill=black, minimum size=2mm] at (ML) {};
  \node[fill=black, minimum size=1mm] at (MR) {};
\end{tikzpicture}
            % Requires: \usepackage{tikz} \usetikzlibrary{patterns}
\begin{tikzpicture}[scale=0.35]
  % --- Tunable geometry (kept simple to match your style) ---
  \def\xTop{2.6}   % half-length of the top edge of the base polygon
  \def\yTop{0.9}   % y of the top edge
  \def\xBot{1.8}   % half-length of the bottom edge
  \def\yBot{-2.0}  % y of the bottom edge
  \def\xGap{0.95}  % half-gap between the two roof wedges along the top edge
  \def\hA{3.1}     % apex height

  % --- Base polygon (trapezoid-like pentagon outline as shown) ---
  \coordinate (L)  at (-\xTop,\yTop);   % top-left
  \coordinate (R)  at ( \xTop,\yTop);   % top-right
  \coordinate (BR) at ( \xBot,\yBot);   % bottom-right
  \coordinate (BL) at (-\xBot,\yBot);   % bottom-left

  % --- Roof geometry ---
  \coordinate (A)  at (0,\hA);          % apex
  \coordinate (ML) at (-\xGap,\yTop);   % inner left point on top edge
  \coordinate (MR) at ( \xGap,\yTop);   % inner right point on top edge
  \coordinate (LM) at ( 0,0);   % inner right point on top edge

  % --- Shaded wedges (crosshatch to match reference figures) ---
  \path[pattern=crosshatch dots, pattern color=gray!55] (A) -- (L) -- (ML) -- cycle;
  \path[pattern=crosshatch dots, pattern color=gray!55] (A) -- (R) -- (MR) -- cycle;

  \path[pattern=crosshatch dots, pattern color=gray!55] (BL) -- (LM) -- (ML) -- (L) -- cycle;
  \path[pattern=crosshatch dots, pattern color=gray!55] (BL) -- (BR) -- (LM) -- cycle;
  \path[pattern=crosshatch dots, pattern color=gray!55] (BR) -- (LM) -- (MR) -- (R) -- cycle;
  
  % --- Outline: thick black edges like in your snippets ---
  \draw[line width=0.5mm, black] (R) -- (BR) -- (BL) -- (L);  % base polygon
  \draw[line width=0.5mm, black] (A) -- (L);
  \draw[line width=0.5mm, black] (A) -- (R);
  \draw[line width=0.5mm, black] (A) -- (ML);
  \draw[line width=0.5mm, black] (A) -- (MR);
  \draw[line width=0.5mm, black] (L) -- (ML); % top edge drawn last for clean bases
  \draw[line width=0.5mm, black] (R) -- (MR); % top edge drawn last for clean bases

  \draw[line width=0.5mm, black] (MR) -- (LM);
  \draw[line width=0.5mm, black] (ML) -- (LM);
  \draw[line width=0.5mm, black] (BR) -- (LM);
  \draw[line width=0.5mm, black] (BL) -- (LM);

  % --- Edge labels ---
  \node at (-3.0,-0.3) {$e_{1}$};
  \node at (0,-2.55) {$e_{2}$};
  \node at ( 3.0,-0.3) {$e_{3}$};

      % shapes of nodes

  \node[star, star points=5, star point ratio=0.5, draw=red, fill=red, minimum size=1mm] at (ML) {};
  \node[star, star points=5, star point ratio=0.5, draw=red, fill=red, minimum size=1mm] at (MR) {};
  \node[star, star points=5, star point ratio=0.5, draw=red, fill=red, minimum size=1mm] at (LM) {};
  
  \node[fill=black, minimum size=2mm] at (A) {};
  \node[fill=black, minimum size=2mm] at (L) {};
  \node[fill=black, minimum size=2mm] at (R) {};
  \node[fill=black, minimum size=2mm] at (BL) {};
  \node[fill=black, minimum size=2mm] at (BR) {};
  % \node[fill=black, minimum size=2mm] at (ML) {};
  % \node[fill=black, minimum size=1mm] at (MR) {};
\end{tikzpicture}
            % Requires: \usepackage{tikz} \usetikzlibrary{patterns}
\begin{tikzpicture}[scale=0.35]
  % --- Tunable geometry (kept simple to match your style) ---
  \def\xTop{2.6}   % half-length of the top edge of the base polygon
  \def\yTop{0.9}   % y of the top edge
  \def\xBot{1.8}   % half-length of the bottom edge
  \def\yBot{-2.0}  % y of the bottom edge
  \def\xGap{0.95}  % half-gap between the two roof wedges along the top edge
  \def\hA{3.1}     % apex height

  % --- Base polygon (trapezoid-like pentagon outline as shown) ---
  \coordinate (L)  at (-\xTop,\yTop);   % top-left
  \coordinate (R)  at ( \xTop,\yTop);   % top-right
  \coordinate (BR) at ( \xBot,\yBot);   % bottom-right
  \coordinate (BL) at (-\xBot,\yBot);   % bottom-left

  % --- Roof geometry ---
  \coordinate (A)  at (0,\hA);          % apex
  \coordinate (ML) at (-\xGap,\yTop);   % inner left point on top edge
  \coordinate (MR) at ( \xGap,\yTop);   % inner right point on top edge
  \coordinate (LM) at ( 0,0); 
  
  \coordinate (LR) at ( 0.5,0); 
  \coordinate (LL) at ( -0.5,0); 

  % --- Shaded wedges (crosshatch to match reference figures) ---
  \path[pattern=crosshatch dots, pattern color=gray!55] (A) -- (L) -- (ML) -- cycle;
  \path[pattern=crosshatch dots, pattern color=gray!55] (A) -- (R) -- (MR) -- cycle;

  \path[pattern=crosshatch dots, pattern color=gray!55] (BL) -- (ML) -- (L) -- cycle;
  \path[pattern=crosshatch dots, pattern color=gray!55] (BR) -- (LR) -- (MR) -- (R) -- cycle;
  \path[pattern=crosshatch dots, pattern color=gray!55] (BR) -- (BL) -- (ML) -- (LR) -- cycle;
  
  % --- Outline: thick black edges like in your snippets ---
  \draw[line width=0.5mm, black] (R) -- (BR) -- (BL) -- (L);  % base polygon
  \draw[line width=0.5mm, black] (A) -- (L);
  \draw[line width=0.5mm, black] (A) -- (R);
  \draw[line width=0.5mm, black] (A) -- (ML);
  \draw[line width=0.5mm, black] (A) -- (MR);
  \draw[line width=0.5mm, black] (L) -- (ML); % top edge drawn last for clean bases
  \draw[line width=0.5mm, black] (R) -- (MR); % top edge drawn last for clean bases

  \draw[line width=0.5mm, black] (ML) -- (BL);
  \draw[line width=0.5mm, black] (ML) -- (LR);
  \draw[line width=0.5mm, black] (LR) -- (BR);
  \draw[line width=0.5mm, black] (MR) -- (LR);

  % --- Edge labels ---
  \node at (-3.0,-0.3) {$e_{1}$};
  \node at (0,-2.55) {$e_{2}$};
  \node at ( 3.0,-0.3) {$e_{3}$};

      % shapes of nodes

  \node[star, star points=5, star point ratio=0.5, draw=red, fill=red, minimum size=1mm] at (ML) {};
  \node[star, star points=5, star point ratio=0.5, draw=red, fill=red, minimum size=1mm] at (LR) {};
  
  \node[fill=black, minimum size=2mm] at (A) {};
  \node[fill=black, minimum size=2mm] at (L) {};
  \node[fill=black, minimum size=2mm] at (R) {};
  \node[fill=black, minimum size=2mm] at (BL) {};
  \node[fill=black, minimum size=2mm] at (BR) {};
  % \node[fill=black, minimum size=2mm] at (ML) {};
  % \node[fill=black, minimum size=1mm] at (MR) {};
\end{tikzpicture}
            % Requires: \usepackage{tikz} \usetikzlibrary{patterns}
\begin{tikzpicture}[scale=0.35]
  % --- Tunable geometry (kept simple to match your style) ---
  \def\xTop{2.6}   % half-length of the top edge of the base polygon
  \def\yTop{0.9}   % y of the top edge
  \def\xBot{1.8}   % half-length of the bottom edge
  \def\yBot{-2.0}  % y of the bottom edge
  \def\xGap{0.95}  % half-gap between the two roof wedges along the top edge
  \def\hA{3.1}     % apex height

  % --- Base polygon (trapezoid-like pentagon outline as shown) ---
  \coordinate (L)  at (-\xTop,\yTop);   % top-left
  \coordinate (R)  at ( \xTop,\yTop);   % top-right
  \coordinate (BR) at ( \xBot,\yBot);   % bottom-right
  \coordinate (BL) at (-\xBot,\yBot);   % bottom-left

  % --- Roof geometry ---
  \coordinate (A)  at (0,\hA);          % apex
  \coordinate (ML) at (-\xGap,\yTop);   % inner left point on top edge
  \coordinate (MR) at ( \xGap,\yTop);   % inner right point on top edge
  \coordinate (LM) at ( 0,0); 
  
  \coordinate (LR) at ( 0.5,0); 
  \coordinate (LL) at ( -0.5,0); 

  % --- Shaded wedges (crosshatch to match reference figures) ---
  \path[pattern=crosshatch dots, pattern color=gray!55] (A) -- (L) -- (ML) -- cycle;
  \path[pattern=crosshatch dots, pattern color=gray!55] (A) -- (R) -- (MR) -- cycle;

  \path[pattern=crosshatch dots, pattern color=gray!55] (BL) -- (LL) -- (ML) -- (L) -- cycle;
  \path[pattern=crosshatch dots, pattern color=gray!55] (BR) -- (LR) -- (MR) -- (R) -- cycle;
  \path[pattern=crosshatch dots, pattern color=gray!55] (BR) -- (BL) -- (LL) -- (LR) -- cycle;
  
  % --- Outline: thick black edges like in your snippets ---
  \draw[line width=0.5mm, black] (R) -- (BR) -- (BL) -- (L);  % base polygon
  \draw[line width=0.5mm, black] (A) -- (L);
  \draw[line width=0.5mm, black] (A) -- (R);
  \draw[line width=0.5mm, black] (A) -- (ML);
  \draw[line width=0.5mm, black] (A) -- (MR);
  \draw[line width=0.5mm, black] (L) -- (ML); % top edge drawn last for clean bases
  \draw[line width=0.5mm, black] (R) -- (MR); % top edge drawn last for clean bases

  \draw[line width=0.5mm, black] (ML) -- (LL);
  \draw[line width=0.5mm, black] (LL) -- (LR);
  \draw[line width=0.5mm, black] (LL) -- (BL);
  \draw[line width=0.5mm, black] (LR) -- (BR);
  \draw[line width=0.5mm, black] (LR) -- (MR);

  % --- Edge labels ---
  \node at (-3.0,-0.3) {$e_{1}$};
  \node at (0,-2.55) {$e_{2}$};
  \node at ( 3.0,-0.3) {$e_{3}$};

      % shapes of nodes

  \node[star, star points=5, star point ratio=0.5, draw=red, fill=red, minimum size=1mm] at (LL) {};
  \node[star, star points=5, star point ratio=0.5, draw=red, fill=red, minimum size=1mm] at (LR) {};
  
  \node[fill=black, minimum size=2mm] at (A) {};
  \node[fill=black, minimum size=2mm] at (L) {};
  \node[fill=black, minimum size=2mm] at (R) {};
  \node[fill=black, minimum size=2mm] at (BL) {};
  \node[fill=black, minimum size=2mm] at (BR) {};
  \node[circle, fill=black, minimum size=2mm] at (ML) {};
  \node[circle, fill=black, minimum size=1mm] at (MR) {};
\end{tikzpicture}
            \caption{The five cases (up to symmetry) for the non-pentagonal facets adjacent to $e_1, e_2$ and $e_3$}
            \label{fig:manypentagons}
        \end{subfigure}
        \caption{Illustrations of the Schlegel diagrams used in the proof of~\cref{thm:3polytope-degree}.}
    \end{figure}
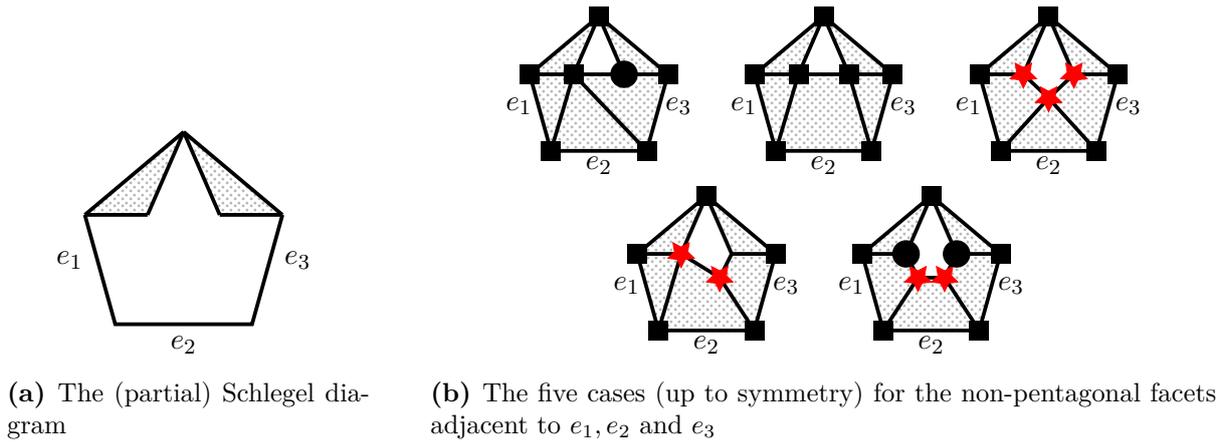
    
    In the notation of~\cref{fig:pentagon1}, $e_1, e_2, e_3$ are each contained in one additional facet, each of which can be a triangle or a square, leading to the five cases shown in~\cref{fig:manypentagons}. 

    The vertices with squares drawn on them cannot increase their degree. The polygons are pieces that are not yet fixed by our considerations and could perhaps be further subdivided. The vertices with a red star are adjacent and only one of their degrees can increase or one of the edges incident to them will become negative. Since the graph of the polytope is three connected, the addition of a vertex in the white region will create edges incident to three of the vertices of the region (or the ones connecting could be removed to disconnect the graph). It follows that in the first four cases no additional vertex can be added. In the last case if there are additional vertices, then there is exactly one new facet containing the additional facet, it can be a triangle, a square or a pentagon. Analyzing those cases we see that no more than 3 vertices can be added.
    
    Assuming $P$ contains no pentagons or vertices of degree $5$, then we have the following linear equations: $d_3+d_4 =f_0$, $p_3+p_4=f_2$, $3d_3+4d_4 = 2f_1 = 3p_3 +4p_4$, which, taken together with Euler's formula, results in a system of linear equations, with $5$ equations and $7$ unknowns; looking for positive integral solutions, it must hold that $d_3 = 8-p_3$, and since $d_3$ is even and nonnegative, we must have $p_3 \in\{0,2,4,6,8\}$. 
    
    Furthermore, notice that if $F_1$ and $F_2$ are quadrilateral faces sharing an edge, then at least one of the two vertices adjacent to the edge has degree 3. From this one obtains that for every face that is a quadrilateral, the number of edges incident to triangles plus the number of vertices of degree three is at least $3$. Each vertex of degree three and each triangle is adjacent to at most 3 quadrilaterals, meaning that $3p_4\le 3(p_3+d_3)= 24$, so $p_4\le 8$ and by duality $d_4\le 8$. Then $f_0 = d_3+d_4\le 8+8 = 16$. We reiterate that the remaining two cases for the degrees of the vertices occurring in the pentagon follow similarly, and the claim follows.
\end{proof}

\Cref{thm:3polytope-degree} allows us to identify many Forman--Ricci-positive $3$-polytopes by scanning the family of planar, $3$-connected graphs for Forman--Ricci positivity. We carried out such an experiment on all such graphs up to and including twelve vertices and found $109$ Forman--Ricci-positive polyhedra. We did this by generating all polyhedral graphs up to this threshold using the software \texttt{plantri} (see~\cite{McKayGraphs,OEIS_A000944,plantri}), and then running each graph through a Python method to compute its curvature. We illustrate the graphs of each of a random sample of 49 such $3$-polytopes in~\cref{fig:forman_positives}.\footnote{Our code is publicly available at \href{https://github.com/jeddyhub/discrete-curvatures-and-convex-polytopes}{https://github.com/jeddyhub/discrete-curvatures-and-convex-polytopes}.} 

\begin{figure}[t!]
    \centering
    \includegraphics[width=0.75\textwidth]{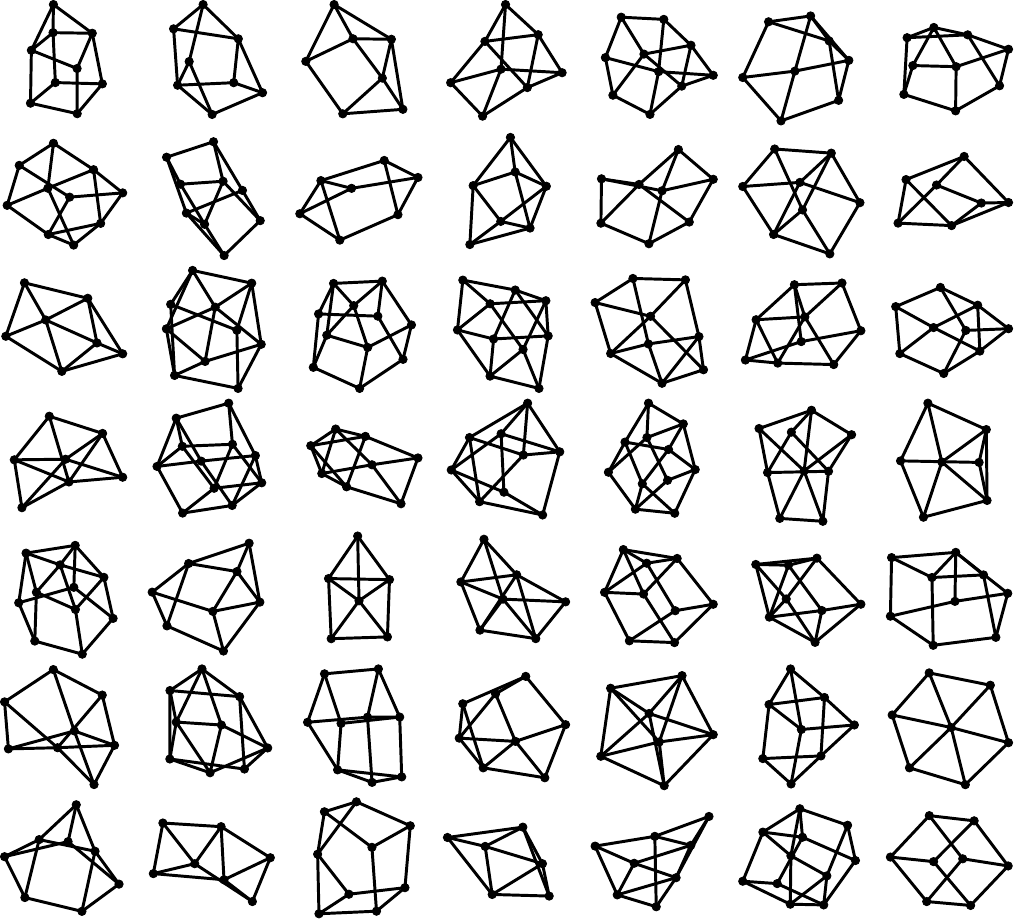}
    \caption{An illustration of the graphs of 49 Forman--Ricci-positive $3$-polytopes, obtained as a random sample of the Forman--Ricci-positive $3$-polytopes on at most twelve vertices.}
    \label{fig:forman_positives}
\end{figure}

\subsection{Forman--Ricci curvature in Dimension $4$}\label{subsec:forman-dim4}

The goal of this section is to show that there can only be finitely many Forman--Ricci-positive $4$-polytopes. The proof analyzes the neighborhood of a vertex of very high degree, to show that an edge connected to it must be negative. 

\begin{theorem}\label{thm:4-dimFiniteness}
  There are finitely many Forman--Ricci-Positive $4$-polytopes.   
\end{theorem}

\begin{proof}
    Assume that $P$ is Forman--Ricci-positive and contains a vertex $v$ whose degree $\delta$ satisfies $\delta\ge 13$. Then the vertex figure $Q$ of $v$ is a $3$-polytope, hence it must have a vertex of degree no larger than $5$. That vertex corresponds to an edge of $P$ that connects $v$ to another vertex $w$. We claim that the edge $e=\{v,w\}$ is not Forman--Ricci-Positive. 

    To bound the curvature at $e$, notice that the positive contribution is $2+k$, where $k$ is the number of $2$-faces that contain $e$. Thus $k$ is exactly the degree of $w$ in $Q$, which is at most $5$. Furthermore, any $2$-face containing $e$ contains exactly one additional edge incident to $v$, so there are $\delta-6$ parallel edges to $e$ that are incident to $v$. It follows that $\mathcal{F}(e) = 2+k - \#\{\text{parallel edges}\} \le 2+5 - (\delta-6) = 13- \delta \le 0$, a contradiction.
    
    It follows that the maximum degree (and hence the average degree) of a Forman--Ricci-Positive $4$-dimensional polytope is at most 12, and thus by~\cref{thm:average-degree}, there are only finitely many such polytopes. 
\end{proof}

\begin{remark}
    Notice that the upper bound does not work for polytopes with non-negative Forman--Ricci curvature, but the proof above shows that any nonnegative $4$-polytope has all of its vertices of degree no more than 13. 
\end{remark}

\subsection{Curvature of simplicial polytopes and 5-dimensional polytopes}\label{subsec:simplicial-five}

In this setting we investigate the class of Forman--Ricci-positive $5$-polytopes and show that  there are finitely many Forman--Ricci-positive simplicial $5$-polytopes. We conclude with a conjecture concerning the extension of ~\cref{thm:FormanFinite3} and~\cref{thm:4-dimFiniteness} to the setting of $d=5$.

We begin with a lemma for computing the average curvature in a simplicial polytope, which specializes~\cref{lemma:AverageCurvatureIdentity}.

\begin{lemma}\label{lemma:simplicial-AverageCurvatureIdentity}
    Let $P$ be a simplicial $d$-polytope. The following equation holds: 
    \begin{eqnarray*}
        \mathcal{K}(P) &=& \frac{1}{f_1}\left(18f_2 + 4f_1 - 2\sum_{k\ge d} k^2d_k - 9f_2\right) \\ &=& \frac{1}{f_1}\left(9f_2+4f_1-\sum_{k\ge d}k^2d_k\right)
    \end{eqnarray*}
\end{lemma}

\begin{proof}
    Let $P$ be a fixed simplicial $d$-polytope. Then it follows that $f_{02}= 3f_2$. Moreover, all $2$-dimensional faces are triangles and hence $\sum_{k\ge3}k^2p_k= 9f_1$. The claim follows from the proof of~\cref{lemma:AverageCurvatureIdentity}. 
\end{proof}

We will use this as a tool to show non-positivity in several instances. The main difficulty in dealing with the expression above concerns the term $\sum_{k\ge d}k^2d_k$. The degree sequences in simplicial polytopes can vary extensively. Nonetheless, the Cauchy-Schwartz inequality implies that 
    \begin{equation*}
        \sum_{k\ge d} k^2d_k \ge  \frac{\left(\sum_{k\ge d}kd_k\right)^2}{\sum_{k\ge d} d_k} = \frac{4f_1^2}{f_0}.
    \end{equation*}
Thus we have:
    \begin{align}\label{eqn:simplicialBound}
        \mathcal{K}(P)f_1 &\le  9f_2+4f_1 - \frac{4f_1^2}{f_0} .
    \end{align}
 The $g$-theorem implies that the right hand expression can be positive in many cases as long as the dimension is at least $6$. In dimensions $4$ and $5$, however, we have that $f_2$ is determined by a linear equation in $f_0$ and $f_1$, and the inequality becomes harder to satisfy, since it is a quadratic in $f_1$ (the larger term) with negative principal coefficient.  

From this setup and a short proof we may conclude the following theorem.

\begin{theorem}
   There are finitely many Forman--Ricci-positive simplicial\break $5$-polytopes. 
\end{theorem}

\begin{proof}
    Let $P$ a $5$-dimensional simplicial polytope, then $f_2= 4f_1-10f_0+20$ by the Dehn-Sommerville equations. Plugging this into~\cref{eqn:simplicialBound} yields
        \begin{align}
            \mathcal{K}(P)f_1 &\le 9(4f_1-10f_0+20) + 4f_1 - 4\frac{f_1^2}{f_0} \\
            &= - \frac{4}{f_0}f_1^2 + 40f_1 +180 -90f_0 .
        \end{align}
    Viewed as a quadratic in $f_1$ it only assumes positive values when $f_1$ assumes a value between the two roots: 
        \begin{eqnarray*}
             \frac{-40 \pm\sqrt{40^2 + 16f_0^{-1}(180-90f_0)}}{-8f_0^{-1}} &=& 5f_0 \pm \frac{\sqrt{40^2f_0^2+16f_0(180-90f_0)}}{8}
             \\ &=& 5f_0 \pm \frac{\sqrt{160f_0^2 +2880f_0}}{8}
        \end{eqnarray*}
    For a large value of $f_0$ we have that $\frac{\sqrt{160f_0^2 +2880f_0}}{8} \le 1.6 f_0$ which would mean that the bound is positive in the range $3.4f_0 \le f_1 \le 6.6f_0$. The average vertex degree of a polytope is $\frac{2f_1}{f_0}$, which is therefore bounded above by $13.2$. According to~\cref{cor: finiteAverage}, there are finitely many such polytopes. 
\end{proof}

With the result above in mind, we pose the following conjecture: 

\begin{conjecture}
    There are finitely many Forman--Ricci-positive $5$-polytopes. 
\end{conjecture}

We suspect that if an infinite family of Forman--Ricci-positive polytopes exists, they will be limited to polytopes with dense graphs and small two dimensional faces.

\section{Resistance Curvature}
\label{sec: resistance results}

In this section we investigate the rarity of polytopal graphs which have everywhere positive resistance curvature. We proceed as follows. In~\Cref{subsec:resistance-positivity}, we obtain an upper bound on the resistance of a pair of vertices based on the lengths of paths appearing between its endpoints. We then apply this to obtain lower bounds on the resistance curvature. In~\Cref{subsec:resistance-positive-polyhedra} we apply this setup to obtain a family of constructions of resistance-positive $3$-polytopes. Finally in~\Cref{subsec:resistance-negative} we obtain a degree-based lower bound on the effective resistance of an edge and use this to show that resistance-positive polytopes are close to degree-regular in a weak sense.

\subsection{Resistance bounds via path lengths}\label{subsec:resistance-positivity}
We begin with the following bound on the effective resistance distance between vertices on the graph of a $d$-polytope in terms of the lengths of edge-disjoint paths between them.

\begin{lemma}\label{lem:resistance-bound}
    Let $G=(V,E)$ be any graph, and let $u, v\in V$ be fixed. Assume that for some $k\geq 2$ the vertices $u, v$ admit $k$ edge-disjoint paths $P_1, P_2,\dotsc, P_k$ which begin and end at $u, v$, respectively. Then we have
        \begin{align*}
            r_{uv} \leq \frac{1}{\sum_{s=1}^k |P_s|^{-1}}.
        \end{align*}
\end{lemma}

\begin{proof}
    We will exhibit an $u, v$-flow and calculate its norm as follows. For $1\leq \ell\leq k$, let $\mathbf{J}_\ell:E'\rightarrow\mathbb{R}$ denote the flow which is supported on the edges of $P_\ell$, has constant value $1$ up to changes in sign (depending on the choice of orientation $E'$), and which satisfies $\mathbf{BJ}_\ell = \mathbf{1}_u - \mathbf{1}_v$. For a choice of coefficients $\gamma = (\gamma_1,\gamma_2,\dotsc, \gamma_k)\in\mathbb{R}^k$ with $\gamma_i\geq 0$ and $\sum_i \gamma_i = 1$, define
        \begin{align*}
            \mathbf{J}_\gamma = \sum_{\ell=1}^k \gamma_\ell \mathbf{J}_\ell,
        \end{align*}
    then $\mathbf{J}_\gamma$ is a feasible $u, v$-flow and we have that by~\cref{prop:flow-resistance}, it holds
        \begin{align*}
            r_{uv} \leq \|\mathbf{J}_\gamma\|_2^2 = \sum_{\ell=1}^k \gamma_\ell^2 |P_\ell|,
        \end{align*}
    so to improve this bound we consider the problem
        \begin{align*}
            \begin{cases}
            \text{minimize }\hspace{.25cm}& \sum_{\ell=1}^k \gamma_\ell^2 |P_\ell|\\
            \text{subject to }\hspace{.25cm}& 0\leq \gamma_\ell\leq 1\\
            & \sum_i \gamma_i = 1 .
            \end{cases}
        \end{align*}
    Let $\mathbf{B} = \mathrm{diag}(|P_1|, |P_2|,\dotsc, |P_k|)\in\mathbb{R}^{k\times k}$ denote the diagonal matrix of path lengths. The Lagrange multiplier becomes $2\mathbf{B}\gamma = \lambda\mathbf{1}_k$, i.e., $\gamma_\ell = \frac{\lambda}{2|P_\ell|}$, so that we have 
        \begin{align*}
            \lambda = \frac{2}{\sum_{\ell=1}^k |P_\ell|^{-1}},\text{ and }\gamma_\ell = \frac{1}{|P_\ell| \sum_{s=1}^k |P_s|^{-1}}.
        \end{align*}
    Thus we have
        \begin{align*}
            r_{uv} \leq \frac{1}{\left(\sum_{s=1}^k |P_s|^{-1}\right)^2}\sum_{\ell=1}^k \frac{|P_\ell|}{|P_\ell|^2} \leq \frac{1}{\sum_{s=1}^k |P_s|^{-1}}.
        \end{align*}
\end{proof}

We may then apply~\cref{lem:resistance-bound} to obtain the following corollary on the resistance curvature of a $3$-polytope.

\begin{theorem}\label{thm:curvature-bound}
    Let $G=(V,E)$ be the $1$-skeleton of a simple $3$-polytope. Fix $v\in V$, and let $\mathcal{C
    }(v)$ denote the set of $2$-faces incident to $v$. For each $C\in\mathcal{C}(v)$, let $\ell_C := |E(C)|$ be the length (edge count) of $C$. Then the resistance curvature of $G$ at $v$ satisfies
        \begin{align*}
            \resistance{v} \geq
            1 \;-\; \frac{1}{2}\sum_{C\in\mathcal{C}(v)}
            \frac{\ell_C-1}{
                (\ell_C-1)\!\left(\sum_{C'\in\mathcal{C}(v)}\frac{1}{\ell_{C'}-1}+1\right)-1}.
        \end{align*}
\end{theorem}

\begin{proof}
    Let $u\in V$ be fixed with $v\sim u$ and note that the edge $e = \{u, v\}$ is incident to exactly $2$ of the $2$-faces belonging to $\mathcal{C}(v)$, which without loss of generality we may take to be $C_1, C_{2}$. Note that with the exception of $e$, said $2$-faces are otherwise edge-disjoint. Therefore, we may construct $3$ edge-disjoint paths from $u$ to $v$ as follows: let the first path $P_1$ consist of exactly $e$, and which has length one; and then let $P_2,\dotsc, P_d$ be obtained by traversing the edges of $C_1,C_2$ from $u$ to $v$ and avoiding $e$, and which have lengths $\ell_{C_1}-1, \ell_{C_2}-1$, respectively. By~\cref{lem:resistance-bound}, we have
        \begin{align*}
            r_{uv} \leq \frac{1}{\sum_{s=1}^{3} |P_s|^{-1}} = \frac{1}{1 + \sum_{s=1}^{2}\frac{1}{\ell_{C_s} - 1}}.
        \end{align*}
    By applying this argument to each edge incident to $v$, we have by straightforward manipulation
        \begin{align*}
            \sum_{u\sim v} r_{uv} \leq \sum_{t=1}^{3} \frac{1}{ 1 + \sum_{s\neq t}\frac{1}{\ell_{C_s} - 1}} &= \sum_{t=1}^{3} \frac{\ell_{C_t} - 1}{ (\ell_{C_t} - 1)(\sum_{s=1}^{3} \frac{1}{\ell_{C_s} - 1}+1) - 1}.
        \end{align*}
    The claim follows.
\end{proof}

\begin{remark}\label{rmk:curvature-test}
   ~\cref{thm:curvature-bound} may be used to verify the resistance positivity of a simple polytope based on information about the $2$-face cycle lengths at each vertex, as follows. Letting $P$ be a simple $3$-polytope and $v\in V(P)$ fixed, let $\boldsymbol\ell(v)\in\mathbb{R}^3$ denote the number of edges in each of the three $2$-faces incident to $v$, ordered in descending fashion. Then~\Cref{thm:curvature-bound} shows that $\resistance{v} > 0$ provided it holds:
        \begin{align}\label{eq:test-for-positive}
            \sum_{i=1}^3
                \frac{\boldsymbol\ell_i-1}{
                    (\boldsymbol\ell_i-1)\!\left(\sum_{j=1}^3\frac{1}{\boldsymbol\ell_{j}-1}+1\right)-1} < 2.
        \end{align}
    It is useful to remark that among all vectors $\mathbf{x}\in \{3, 4, 5, 6\}^3$, the only ``forbidden'' vectors for which~\Cref{eq:test-for-positive} fails for the following four vectors:
        \begin{align}\label{eq:forbidden}
            \mathbf{x} = (5, 5, 5), (6, 6, 4), (6, 6, 5), (6, 6, 6).
        \end{align}
    We also remark that~\Cref{eq:test-for-positive} holds for any $3$-tuple of the form $(a, 3, 3)$ where $a\ge 3$.
\end{remark}

\subsection{Resistance-positive polytopes and $\Delta$-expansions}\label{subsec:resistance-positive-polyhedra}

In this subsection we explore examples of graphs of $3$-polytopes which have everywhere-positive resistance curvature. First, we note the following useful fact that establishes the existence of infinitely many $d$-polytopes with positive resistance curvature. 

\begin{theorem}\label{thm:resis-curv-transitive}
    Let $G=(V, E)$ be a vertex transitive graph with $|V| =n$. Then each node $v\in V$ has constant positive resistance curvature which satisfies
        \begin{align*}
            \resistance{v} =\frac{1}{n}.
        \end{align*}
\end{theorem}

The proof of~\Cref{thm:resis-curv-transitive} consists of straightforward linear algebra and the result covers such instances as the platonic solids as well as any polytope with $1$-skeleton isomorphic to the complete graph $K_n$. In dimension three, however, we note that there exist resistance positive families of $3$-polytopes which are not vertex transitive, although their classification seems at present out of reach.

To explore this angle, we can first apply~\Cref{rmk:curvature-test} to uncover a class of resistance positive simple $3$-polytopes which are obtained as $\Delta$-expansions of various $3$-polytopes. Recall that if $P$ is a simple $3$-polytope and $v\in V(P)$, the $\Delta$-expansion of $P$ at $v$ is the simple $3$-polytope obtained by replacing $v$ with a triangle (and which can be visualized as slicing a corner off of the polytope).

\begin{theorem}\label{thm:delta-expansions}
    Let $P$ be a simple $3$-polytope. For each $v\in V(P)$, let $\boldsymbol\ell(v)$ denote the vector of face lengths as defined in~\cref{rmk:curvature-test}. Assume:
        \begin{enumerate}[label=(\roman*)]
            \item $\boldsymbol{\ell}(v)\in \{3, 4, 5\}^3$ for each $v\in V(P)$,
            \item and no $\boldsymbol{\ell}(v)$ equals $(5, 5, 5)$.
        \end{enumerate}
    Let $v_0\in V(P)$ be fixed. Assume further that:
        \begin{enumerate}
            \item[(iii)] At most one of the entries of the vector $\boldsymbol{\ell}(v_0)$ equals five,
            \item[(iv)] for each $u\in V(P)$ with $u\sim v_0$, the $2$-face incident to $u$ which is not incident to $v_0$ contains at most four edges,
            \item[(v)] and that no $w\in V(P)$ belonging to the three faces incident to $v_0$ has a face sequence $(5, 5, 4)$.
        \end{enumerate}
    Then the $\Delta$-expansion of $P$ at $v_0$ is resistance positive.  
\end{theorem}

\begin{proof}
    Assume without loss of generality that the vertices neighboring $v_0$ are labelled $v_1, v_2, v_3$. Let $F_{012}$ denote the face of $P$ containing the vertices $v_0, v_1, v_2$ and similarly for $F_{023}, F_{031}$. Write $\boldsymbol{\ell}(v_0) = (\ell_1, \ell_2, \ell_3)$ where
        \begin{align*}
            \ell_1 = |F_{012}|,\quad \ell_2 = |F_{023}|,\quad \ell_3 = |F_{031}|,
        \end{align*}
    for which we assume without loss of generality that $\ell_3\leq \ell_2\leq \ell_1$. Applying the $\Delta$-expansion of $P$ at $v_0$, vertex $v_0$ is replaced by three new vertices $u_1, u_2, u_3$, which can be taken so that $u_i$ is incident to $v_i$ and the remaining two $u_j$ with $j\neq i$. In this case we have face length sequences
        \begin{align*}
            \boldsymbol{\ell}'(u_1) &= (\ell_1+1, \ell_3+1, 3),\\
            \boldsymbol{\ell}'(u_2) &= (\ell_1+1, \ell_2+1, 3),\\
            \boldsymbol{\ell}'(u_3) &= (\ell_2+1, \ell_3+1, 3).
        \end{align*}
    Here, $\boldsymbol{\ell}'(\cdot)$ denote the face count vector in the $\Delta$-expansion of $P$ to avoid confusion. Since each $\ell_i\leq 5$, it holds that no entry of $\boldsymbol{\ell}(u_i)$ exceeds six, and each such sequence contains an entry of three. Thus by~\Cref{rmk:curvature-test}, $\resistance{u_i} > 0$ for $i=1, 2, 3$. Next consider $v_1$ as it appears in the $\Delta$-expansion of $P$. It follows that
        \begin{align*}
            \boldsymbol{\ell}'(v_1) = (\ell_1+1, \ell_3+1, s)
        \end{align*}
    for some $s\in \{3, 4\}$. Moreover, because at most one entry of $\boldsymbol{\ell}(v_0)$ equals $5$, we cannot have $\ell_1=\ell_3=5$. Thus, again by~\Cref{rmk:curvature-test}, we must have $\resistance{v_1} > 0$. The same argument applies to the case of $v_2, v_3$ and it follows that $\resistance{v_i} > 0$ for each $i$.

    Finally, let $w\in V(P)\setminus\{v_0,v_1,v_2,v_3\}$ be fixed. If the vertex $w$ and its incident facets are untouched by the $\Delta$-expansion, $\resistance{w} > 0$ automatically by assumptions~(i) and~(ii) via~\Cref{rmk:curvature-test}.
    
    In the $\Delta$-expansion only the faces $F_{012},F_{023},F_{031}$ change, each increasing its length by exactly $1$. Any such $w$ lies on at most one of these three faces (otherwise $w$ would be one of $v_1,v_2,v_3$), so $\boldsymbol{\ell}'(w)$ is obtained from $\boldsymbol{\ell}(w)\in\{3,4,5\}^3$ by increasing at most one coordinate by $1$. By assumptions (i), (ii), and (v), $\boldsymbol{\ell}'(w)$ also avoids the four vectors in~\eqref{eq:forbidden} and thus $\resistance{w}>0$. This proves the theorem.
\end{proof}

Note that conditions \emph{(i)-(v)} in~\cref{thm:delta-expansions} are sufficient but not necessary; one can exhibit constructions failing one or more of the aforementioned criteria but which still determine resistance positive $3$-polytopes. In~\Cref{fig:delta-expansions} we illustrate resistance-positive $3$-polytopes and their $\Delta$-expansions.

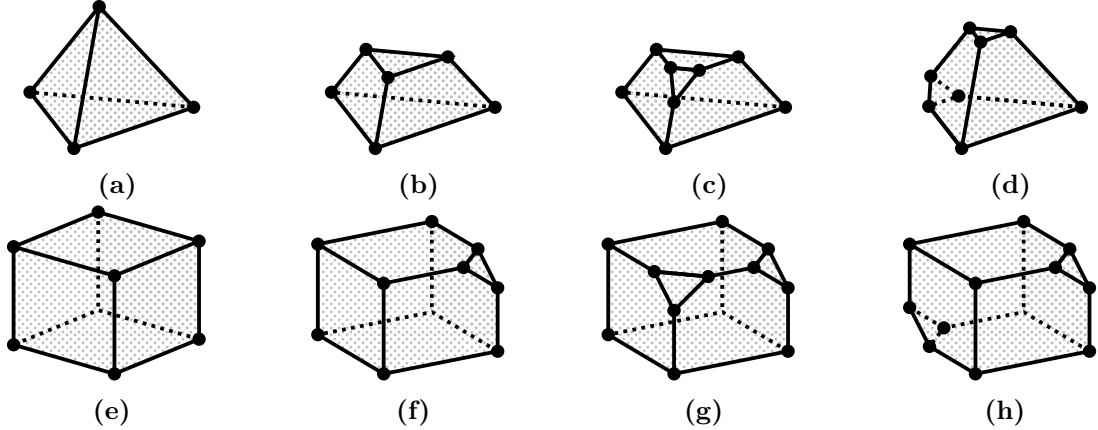
\begin{figure}[t!]
    \centering
    \sc{$3$-Polytopes and $\Delta$-expansions}
    
    \begin{subfigure}[t]{0.24\textwidth}
        \centering
        % (1) A 3-simplex (tetrahedron) in the same oblique TikZ style
\begin{tikzpicture}[x={(0.866cm,0.30cm)}, y={(-0.866cm,0.30cm)}, z={(0cm,1cm)}, scale=1.0]
  \def\r{1.05}
  \def\h{1.45}
  \def\rot{-150}

  % --- Base triangle and apex ---
  \foreach \i in {1,...,3}{
    \pgfmathsetmacro{\ang}{\rot + 120*(\i-1)}
    \coordinate (B\i) at ({\r*cos(\ang)},{\r*sin(\ang)},{0});
  }
  \coordinate (A) at (0,0,\h);

  % --- Shaded faces (front-ish) ---
  \path[pattern=crosshatch dots, pattern color=gray!50] (B1)--(B2)--(B3)--cycle;         % base
  \path[pattern=crosshatch dots, pattern color=gray!50] (A)--(B1)--(B2)--cycle;          % front side
  \path[pattern=crosshatch dots, pattern color=gray!50] (A)--(B2)--(B3)--cycle;          % right/front side

  % --- Edges (solid = front, dotted = hidden) ---
  % base rim
  \draw[line width=0.5mm] (B1)--(B2);
  \draw[line width=0.5mm, dotted] (B2)--(B3);
  \draw[line width=0.5mm] (B3)--(B1);
  % sides to apex
  \draw[line width=0.5mm] (A)--(B1);
  \draw[line width=0.5mm] (A)--(B2);
  \draw[line width=0.5mm] (A)--(B3);

  % --- Vertex markers ---
  \foreach \p in {A,B1,B2,B3}{
    \node[circle, fill=black, inner sep=0pt, minimum size=5pt] at (\p) {};
  }
\end{tikzpicture}
        \caption{}\label{subfig:delta-expansions-1}
    \end{subfigure}
    \begin{subfigure}[t]{0.24\textwidth}
        \centering
        % (2) Delta expansion of (1): shave apex A -> new triangular face U1U2U3
\begin{tikzpicture}[x={(0.866cm,0.30cm)}, y={(-0.866cm,0.30cm)}, z={(0cm,1cm)}, scale=1.0]
  \def\r{1.05}
  \def\h{1.45}
  \def\rot{-150}
  \def\tcut{0.5} % truncation depth toward the base vertices

  % --- Base triangle and apex (for locating truncation points) ---
  \foreach \i in {1,...,3}{
    \pgfmathsetmacro{\ang}{\rot + 120*(\i-1)}
    \coordinate (B\i) at ({\r*cos(\ang)},{\r*sin(\ang)},{0});
  }
  \coordinate (A) at (0,0,\h);

  % --- New vertices on AB-edges after shaving A ---
  \coordinate (U1) at ($(A)!\tcut!(B1)$);
  \coordinate (U2) at ($(A)!\tcut!(B2)$);
  \coordinate (U3) at ($(A)!\tcut!(B3)$);

  % --- Shaded faces ---
  \path[pattern=crosshatch dots, pattern color=gray!50] (B1)--(B2)--(B3)--cycle;                 % base
  \path[pattern=crosshatch dots, pattern color=gray!50] (B1)--(B2)--(U2)--(U1)--cycle;           % face from A-B1-B2
  \path[pattern=crosshatch dots, pattern color=gray!50] (B2)--(B3)--(U3)--(U2)--cycle;           % face from A-B2-B3
  % new triangular face at truncation
  \path[pattern=crosshatch dots, pattern color=gray!50] (U1)--(U2)--(U3)--cycle;

  % --- Edges ---
  % base rim
  \draw[line width=0.5mm] (B1)--(B2);
  \draw[line width=0.5mm, dotted] (B2)--(B3);
  \draw[line width=0.5mm] (B3)--(B1);
  % slanted edges from new points to base
  \draw[line width=0.5mm] (U1)--(B1);
  \draw[line width=0.5mm] (U2)--(B2);
  \draw[line width=0.5mm] (U3)--(B3);
  % boundary of the new triangular face
  \draw[line width=0.5mm] (U1)--(U2)--(U3)--(U1);

  % --- Vertex markers ---
  \foreach \p in {B1,B2,B3,U1,U2,U3}{
    \node[circle, fill=black, inner sep=0pt, minimum size=5pt] at (\p) {};
  }
\end{tikzpicture}
        \caption{}\label{subfig:delta-expansions-2}
    \end{subfigure}%
    \begin{subfigure}[t]{0.24\textwidth}
        \centering
        % (3) Delta expansion of (2): shave ONE newly created corner (take U1)
\begin{tikzpicture}[x={(0.866cm,0.30cm)}, y={(-0.866cm,0.30cm)}, z={(0cm,1cm)}, scale=1.0]
  \def\r{1.05}
  \def\h{1.45}
  \def\rot{-150}
  \def\tcut{0.5}   % first truncation at A
  \def\tcuttwo{0.35}% second truncation at U1

  % --- Base and first truncation (same as fig. 2) ---
  \foreach \i in {1,...,3}{
    \pgfmathsetmacro{\ang}{\rot + 120*(\i-1)}
    \coordinate (B\i) at ({\r*cos(\ang)},{\r*sin(\ang)},{0});
  }
  \coordinate (A) at (0,0,\h);
  \coordinate (U1) at ($(A)!\tcut!(B1)$);
  \coordinate (U2) at ($(A)!\tcut!(B2)$);
  \coordinate (U3) at ($(A)!\tcut!(B3)$);

  % --- Second truncation near U1 (along U1->B1, U1->U2, U1->U3) ---
  \coordinate (W1B) at ($(U1)!\tcuttwo!(B1)$);
  \coordinate (W1U2) at ($(U1)!\tcuttwo!(U2)$);
  \coordinate (W1U3) at ($(U1)!\tcuttwo!(U3)$);

  % --- Shaded faces (front-ish) ---
  \path[pattern=crosshatch dots, pattern color=gray!50] (B1)--(B2)--(B3)--cycle;                        % base
  \path[pattern=crosshatch dots, pattern color=gray!50] (B1)--(B2)--(U2)--(W1U2)--(W1B)--cycle;         % modified (A-B1-B2)
  \path[pattern=crosshatch dots, pattern color=gray!50] (B3)--(B1)--(W1B)--(W1U3)--(U3)--cycle;         % modified (A-B3-B1)
  % small new triangle at U1
  \path[pattern=crosshatch dots, pattern color=gray!50] (W1B)--(W1U2)--(W1U3)--cycle;
  \path[pattern=crosshatch dots, pattern color=gray!50] (U2)--(U3)--(W1U2)--(W1U3)--cycle;

  % --- Edges ---
  % base rim
  \draw[line width=0.5mm] (B1)--(B2);
  \draw[line width=0.5mm, dotted] (B2)--(B3);
  \draw[line width=0.5mm] (B3)--(B1);

  % slanted/base connections
  \draw[line width=0.5mm] (W1B)--(B1);
  \draw[line width=0.5mm] (U2)--(B2);
  \draw[line width=0.5mm] (U3)--(B3);

  % former truncation triangle now clipped at U1
  \draw[line width=0.5mm] (U2)--(U3);
  \draw[line width=0.5mm] (W1U2)--(W1U3);

  % small triangular boundary at U1
  \draw[line width=0.5mm] (W1B)--(W1U2)--(W1U3)--(W1B);

  % connectors (visible/hidden) on clipped edges
  \draw[line width=0.5mm] (U2)--(W1U2);
  \draw[line width=0.5mm] (U3)--(W1U3);

  % --- Vertex markers ---
  \foreach \p in {B1,B2,B3,U2,U3,W1B,W1U2,W1U3}{
    \node[circle, fill=black, inner sep=0pt, minimum size=5pt] at (\p) {};
  }
\end{tikzpicture}
        \caption{}\label{subfig:delta-expansions-3}
    \end{subfigure}
    \begin{subfigure}[t]{0.24\textwidth}
        \centering
        % (4) Delta expansion of (2): ALSO shave ONE ORIGINAL corner (take B3)
\begin{tikzpicture}[x={(0.866cm,0.30cm)}, y={(-0.866cm,0.30cm)}, z={(0cm,1cm)}, scale=1.0]
  \def\r{1.05}
  \def\h{1.45}
  \def\rot{-150}
  \def\tcut{0.25}      % truncation depth for apex A
  \def\tcutB{0.25}     % truncation depth at original base corner B3

  % --- Base triangle and apex ---
  \foreach \i in {1,...,3}{
    \pgfmathsetmacro{\ang}{\rot + 120*(\i-1)}
    \coordinate (B\i) at ({\r*cos(\ang)},{\r*sin(\ang)},{0});
  }
  \coordinate (A) at (0,0,\h);

  % --- First truncation: shave A -> points on AB-edges ---
  \coordinate (U1) at ($(A)!\tcut!(B1)$);
  \coordinate (U2) at ($(A)!\tcut!(B2)$);
  \coordinate (U3) at ($(A)!\tcut!(B3)$);

  % --- Second truncation: shave ORIGINAL corner B3 ---
  \coordinate (V31) at ($(B3)!\tcutB!(B1)$);
  \coordinate (V32) at ($(B3)!\tcutB!(B2)$);
  \coordinate (V3U) at ($(B3)!\tcutB!(U3)$);

  % --- Shaded faces (crosshatch like the style guides) ---
  % Base (with B3 corner removed)
  \path[pattern=crosshatch dots, pattern color=gray!50]
    (B1)--(B2)--(V32)--(V31)--cycle;

  % Front face from (A-B1-B2) after A-truncation (unchanged by B3 shave)
  \path[pattern=crosshatch dots, pattern color=gray!50]
    (B1)--(B2)--(U2)--(U1)--cycle;

  % Right/front face from (A-B2-B3) after both truncations
  \path[pattern=crosshatch dots, pattern color=gray!50]
    (B2)--(V32)--(V3U)--(U3)--(U2)--cycle;

  % New triangular faces from the two truncations
  \path[pattern=crosshatch dots, pattern color=gray!50] (U1)--(U2)--(U3)--cycle;        % at A
  \path[pattern=crosshatch dots, pattern color=gray!50] (V31)--(V32)--(V3U)--cycle;      % at B3

  % --- Edges (solid = front, dotted = hidden) ---
  % Base rim after B3 truncation
  \draw[line width=0.5mm] (B1)--(B2);
  \draw[line width=0.5mm, dotted] (B2)--(V32)--(V31);
  \draw[line width=0.5mm] (V31)--(B1);

  % Slanted/base connections from A-truncation points
  \draw[line width=0.5mm] (U1)--(B1);
  \draw[line width=0.5mm] (U2)--(B2);
  \draw[line width=0.5mm] (U3)--(V3U); % replaces (U3)--(B3)

  % Boundaries of the new triangular faces
  % \draw[line width=0.5mm] (U1)--(U2)--(U3)--(U1);
    \draw[line width=0.5mm] (U1)--(U2);
    \draw[line width=0.5mm] (U2)--(U3);
    \draw[line width=0.5mm] (U3)--(U1);

      %\draw[line width=0.5mm] (V31)--(V32);
        \draw[line width=0.5mm, dotted] (V3U)--(V32);
      \draw[line width=0.5mm] (V3U)--(V31);

  % \draw[line width=0.5mm] (V31)--(V32)--(V3U)--(V31);

  % Connectors from B3-truncation to neighbors
  \draw[line width=0.5mm] (V31)--(B1);
  \draw[line width=0.5mm, dotted] (V32)--(B2);
  \draw[line width=0.5mm, dotted] (V3U)--(U3);

  % --- Vertex markers ---
  \foreach \p in {B1,B2,U1,U2,U3,V31,V32,V3U}{
    \node[circle, fill=black, inner sep=0pt, minimum size=5pt] at (\p) {};
  }
\end{tikzpicture}
        \caption{}\label{subfig:delta-expansions-4}
    \end{subfigure}%
    
    \begin{subfigure}[t]{0.24\textwidth}
        \centering
        % (1) Cube polyhedron in the requested style
\begin{tikzpicture}[x={(0.866cm,0.30cm)}, y={(-0.866cm,0.30cm)}, z={(0cm,1cm)}, scale=1.0]
  \def\r{1.0}
  \def\h{1.3}
  \def\rot{-130}

  % --- Coordinates ---
  \foreach \i in {1,...,4}{
    \pgfmathsetmacro{\ang}{\rot + 90*(\i-1)}
    \coordinate (T\i) at ({\r*cos(\ang)},{\r*sin(\ang)},{\h});
    \coordinate (B\i) at ({\r*cos(\ang)},{\r*sin(\ang)},{0});
  }

  % --- Shaded faces (front-ish) ---
  \path[pattern=crosshatch dots, pattern color=gray!50]
    (T1)--(T2)--(T3)--(T4)--cycle;                         % top
  \path[pattern=crosshatch dots, pattern color=gray!50]
    (B4)--(T4)--(T1)--(B1)--cycle;                         % left front
  \path[pattern=crosshatch dots, pattern color=gray!50]
    (B1)--(T1)--(T2)--(B2)--cycle;                         % front
  \path[pattern=crosshatch dots, pattern color=gray!50]
    (B2)--(T2)--(T3)--(B3)--cycle;                         % right front

  % --- Edges ---
  % Top rim (all solid)
  \draw[line width=0.5mm] (T1)--(T2)--(T3)--(T4)--(T1);

  % Bottom rim (front two solid, back two dotted)
  \draw[line width=0.5mm] (B4)--(B1)--(B2);
  \draw[line width=0.5mm, dotted] (B2)--(B3)--(B4);

  % Vertical edges (front three solid, back-right dotted)
  \draw[line width=0.5mm] (T1)--(B1);
  \draw[line width=0.5mm] (T2)--(B2);
  \draw[line width=0.5mm, dotted] (T3)--(B3);
  \draw[line width=0.5mm] (T4)--(B4);

  % --- Vertex markers ---
  \foreach \i in {1,...,4}{
    \node[circle, fill=black, inner sep=0pt, minimum size=5pt] at (T\i) {};
  }
  \foreach \p in {B4,B1,B2}{
    \node[circle, fill=black, inner sep=0pt, minimum size=5pt] at (\p) {};
  }
\end{tikzpicture}
        \caption{}\label{subfig:delta-expansions-5}
    \end{subfigure}
    \begin{subfigure}[t]{0.24\textwidth}
        \centering
        % (2) Delta expansion of a cube: shave ONE corner (near T2) -> triangular face
\begin{tikzpicture}[x={(0.866cm,0.30cm)}, y={(-0.866cm,0.30cm)}, z={(0cm,1cm)}, scale=1.0]
  \def\r{1.0}
  \def\h{1.2}
  \def\rot{-150}
  \def\tcut{0.3} % how deep the corner is shaved

  % --- Base cube coordinates ---
  \foreach \i in {1,...,4}{
    \pgfmathsetmacro{\ang}{\rot + 90*(\i-1)}
    \coordinate (T\i) at ({\r*cos(\ang)},{\r*sin(\ang)},{\h});
    \coordinate (B\i) at ({\r*cos(\ang)},{\r*sin(\ang)},{0});
  }

  % --- Truncation at vertex T2: points on edges T2->T1, T2->T3, T2->B2 ---
  \coordinate (U21) at ($(T2)!\tcut!(T1)$);
  \coordinate (U23) at ($(T2)!\tcut!(T3)$);
  \coordinate (U2B) at ($(T2)!\tcut!(B2)$);

  % --- Shaded faces ---
  % Top face with corner removed
  \path[pattern=crosshatch dots, pattern color=gray!50]
    (T1)--(U21)--(U23)--(T3)--(T4)--cycle;

  % Left/front/right side faces (with the two front faces updated to pentagons)
  \path[pattern=crosshatch dots, pattern color=gray!50]
    (B4)--(T4)--(T1)--(B1)--cycle;                           % left front
  \path[pattern=crosshatch dots, pattern color=gray!50]
    (B1)--(T1)--(U21)--(U2B)--(B2)--cycle;                   % front (modified)
  \path[pattern=crosshatch dots, pattern color=gray!50]
    (B2)--(U2B)--(U23)--(T3)--(B3)--cycle;                   % right front (modified)

  % New triangular face from the truncation
  \path[pattern=crosshatch dots, pattern color=gray!50]
    (U21)--(U23)--(U2B)--cycle;

  % --- Edges ---
  % Top rim segments (all solid)
  \draw[line width=0.5mm] (T1)--(U21)--(U23)--(T3)--(T4)--(T1);

  % Bottom rim (front two solid, back two dotted)
  \draw[line width=0.5mm] (B4)--(B1)--(B2);
  \draw[line width=0.5mm, dotted] (B2)--(B3)--(B4);

  % Vertical / truncation edges
  \draw[line width=0.5mm] (T1)--(B1);        % front-left
  \draw[line width=0.5mm] (U2B)--(B2);       % truncated vertical at the front-right
  \draw[line width=0.5mm, dotted] (T3)--(B3);
  \draw[line width=0.5mm] (T4)--(B4);

  % Boundary of the new triangular face
  \draw[line width=0.5mm] (U21)--(U23)--(U2B)--(U21);

  % --- Vertex markers ---
  \foreach \p in {T1,T3,T4,U21,U23,U2B,B4,B1,B2}{
    \node[circle, fill=black, inner sep=0pt, minimum size=5pt] at (\p) {};
  }
\end{tikzpicture}
        \caption{}\label{subfig:delta-expansions-5}
    \end{subfigure}%
    \begin{subfigure}[t]{0.24\textwidth}
        \centering
        % (4) Delta expansion of (2) with an ADJACENT corner also shaved (near T1)
\begin{tikzpicture}[x={(0.866cm,0.30cm)}, y={(-0.866cm,0.30cm)}, z={(0cm,1cm)}, scale=1.0]
  \def\r{1.0}
  \def\h{1.2}
  \def\rot{-150}
  \def\tcut{0.3} % how deep the corner is shaved

  % --- Base cube coordinates ---
  \foreach \i in {1,...,4}{
    \pgfmathsetmacro{\ang}{\rot + 90*(\i-1)}
    \coordinate (T\i) at ({\r*cos(\ang)},{\r*sin(\ang)},{\h});
    \coordinate (B\i) at ({\r*cos(\ang)},{\r*sin(\ang)},{0});
  }

  % --- Truncation at T2 (as in (2)) ---
  \coordinate (U21) at ($(T2)!\tcut!(T1)$);
  \coordinate (U23) at ($(T2)!\tcut!(T3)$);
  \coordinate (U2B) at ($(T2)!\tcut!(B2)$);

  % --- Adjacent truncation at T1: points on T1->T2, T1->T4, T1->B1 ---
  \coordinate (A12) at ($(T1)!\tcut!(T2)$);
  \coordinate (A14) at ($(T1)!\tcut!(T4)$);
  \coordinate (A1B) at ($(T1)!\tcut!(B1)$);

  % --- Shaded faces ---
  % Top face with corners T2 and T1 both shaved:
  % traverse around the visible boundary on the top: A14 -> T4 -> T3 -> U23 -> U21 -> A12
  \path[pattern=crosshatch dots, pattern color=gray!50]
    (A14)--(T4)--(T3)--(U23)--(U21)--(A12)--cycle;

  % Left/front/right side faces updated
  \path[pattern=crosshatch dots, pattern color=gray!50]
    (B4)--(T4)--(A14)--(A1B)--(B1)--cycle;                 % left/front (T1 shaved)
  \path[pattern=crosshatch dots, pattern color=gray!50]
    (B1)--(A1B)--(A12)--(U2B)--(B2)--cycle;                 % front (both T1 & T2 shaved)
  \path[pattern=crosshatch dots, pattern color=gray!50]
    (B2)--(U2B)--(U23)--(T3)--(B3)--cycle;                  % right/front (T2 shaved)

  % Triangular faces (front-visible)
  \path[pattern=crosshatch dots, pattern color=gray!50]
    (U21)--(U23)--(U2B)--cycle;                             % at T2
  \path[pattern=crosshatch dots, pattern color=gray!50]
    (A14)--(A12)--(A1B)--cycle;                             % at T1

  % --- Edges ---
  % Top rim segments (solid)
  \draw[line width=0.5mm] (T4)--(A14)--(A12)--(U21)--(U23)--(T3)--(T4);

  % Bottom rim (unchanged from (2): front two solid, back two dotted)
  \draw[line width=0.5mm] (B4)--(B1)--(B2);
  \draw[line width=0.5mm, dotted] (B2)--(B3)--(B4);

  % Vertical / truncation edges
  \draw[line width=0.5mm] (A1B)--(B1);      % truncated vertical at T1
  \draw[line width=0.5mm] (U2B)--(B2);      % truncated vertical at T2
  \draw[line width=0.5mm, dotted] (T3)--(B3);
  \draw[line width=0.5mm] (T4)--(B4);

  % Triangular boundaries
  \draw[line width=0.5mm] (U21)--(U23)--(U2B)--(U21);
  \draw[line width=0.5mm] (A14)--(A12)--(A1B)--(A14);

  % --- Vertex markers ---
  \foreach \p in {A14,A12,U21,U23,T3,T4,A1B,U2B,B1,B2,B4}{
    \node[circle, fill=black, inner sep=0pt, minimum size=5pt] at (\p) {};
  }
\end{tikzpicture}
        \caption{}\label{subfig:delta-expansions-6}
    \end{subfigure}
    \begin{subfigure}[t]{0.24\textwidth}
        \centering
        % (3) Delta expansion of (2) with the OPPOSITE corner also shaved (near B4)
\begin{tikzpicture}[x={(0.866cm,0.30cm)}, y={(-0.866cm,0.30cm)}, z={(0cm,1cm)}, scale=1.0]
  \def\r{1.0}
  \def\h{1.2}
  \def\rot{-150}
  \def\tcut{0.3} % how deep the corner is shaved

  % --- Base cube coordinates ---
  \foreach \i in {1,...,4}{
    \pgfmathsetmacro{\ang}{\rot + 90*(\i-1)}
    \coordinate (T\i) at ({\r*cos(\ang)},{\r*sin(\ang)},{\h});
    \coordinate (B\i) at ({\r*cos(\ang)},{\r*sin(\ang)},{0});
  }

  % --- Truncation at T2 (as before) ---
  \coordinate (U21) at ($(T2)!\tcut!(T1)$);
  \coordinate (U23) at ($(T2)!\tcut!(T3)$);
  \coordinate (U2B) at ($(T2)!\tcut!(B2)$);

  % --- Opposite truncation at B4: points on B4->B1, B4->B3, B4->T4 ---
  \coordinate (W41) at ($(B4)!\tcut!(B1)$);
  \coordinate (W43) at ($(B4)!\tcut!(B3)$);
  \coordinate (W4T) at ($(B4)!\tcut!(T4)$);

  % --- Shaded faces ---
  % Top (with T2 corner removed)
  \path[pattern=crosshatch dots, pattern color=gray!50]
    (T1)--(U21)--(U23)--(T3)--(T4)--cycle;

  % Left/front/right side faces (update left face due to B4 truncation)
  \path[pattern=crosshatch dots, pattern color=gray!50]
    (W4T)--(T4)--(T1)--(B1)--(W41)--cycle;                  % left/front (modified)
  \path[pattern=crosshatch dots, pattern color=gray!50]
    (B1)--(T1)--(U21)--(U2B)--(B2)--cycle;                  % front (modified by T2)
  \path[pattern=crosshatch dots, pattern color=gray!50]
    (B2)--(U2B)--(U23)--(T3)--(B3)--cycle;                  % right/front (modified by T2)

  % New triangular faces (front one visible, back one mostly hidden)
  \path[pattern=crosshatch dots, pattern color=gray!50]
    (U21)--(U23)--(U2B)--cycle;

  % --- Edges ---
  % Top rim (solid)
  \draw[line width=0.5mm] (T1)--(U21)--(U23)--(T3)--(T4)--(T1);

  % Bottom rim (account for B4 truncation)
  \draw[line width=0.5mm] (W41)--(B1)--(B2);                % front two visible
  \draw[line width=0.5mm, dotted] (B2)--(B3)--(W43)--(W41); % the rest hidden

  % Vertical / slanted edges
  \draw[line width=0.5mm] (T1)--(B1);
  \draw[line width=0.5mm] (U2B)--(B2);
  \draw[line width=0.5mm, dotted] (T3)--(B3);
  \draw[line width=0.5mm] (T4)--(W4T);

  % Triangular boundaries
  \draw[line width=0.5mm] (U21)--(U23)--(U2B)--(U21);
    %\draw[line width=0.5mm, dotted] (W41)--(W43);
    \draw[line width=0.5mm, dotted] (W43)--(W4T);
    \draw[line width=0.5mm]  (W4T)--(W41);

  % --- Vertex markers ---
  \foreach \p in {T1,T3,T4,U21,U23,U2B,W41,W43,W4T,B1,B2}{
    \node[circle, fill=black, inner sep=0pt, minimum size=5pt] at (\p) {};
  }
\end{tikzpicture}
        \caption{}\label{subfig:delta-expansions-8}
    \end{subfigure}%
    ~
    \caption{\textit{(a)} The $3$-simplex, vertex transitive and resistance positive.\textit{(b)--(d)} $\Delta$-expansions of \textit{(a)}, resistance-positive via~\Cref{thm:delta-expansions}. \textit{(e)} The $3$-cube, likewise vertex-transitive and resistance positive. \textit{(f)} A $\Delta$-expansion of the $3$-cube that is resistance positive via~\Cref{thm:delta-expansions}. \textit{(g)} A $\Delta$-expansion of the $3$-cube not covered by the sufficient criteria of~\Cref{thm:delta-expansions}, yet containing no forbidden face sequences (cf.~\Cref{eq:forbidden}). \textit{(h)} A $\Delta$-expansion of the $3$-cube failing the same criteria and containing forbidden face sequences (cf.~\Cref{eq:forbidden}) which is resistance positive. 
    }\label{fig:delta-expansions}
\end{figure}

Lastly, we offer a direct construction of a family of non-vertex transitive resistance positive $3$-polytopes which is unbounded in size, as follows. First we recall the Cartesian product of graphs operation: If $G=(V_G, E_G)$ and $H=(V_H, E_H)$ are graphs, $G\times H$ is the graph with vertex set $V_G\times V_H$ and edges
    \begin{align*}
        E(G\times H) &= \{\{ (u_1, v_1), (u_2, v_2)\} \;:\;\{u_1, u_2\}\in E_G\text{ and }v_1 = v_2,\\
        &\qquad\qquad\text{ or }\{v_1, v_2\}\in E_H\text{ and }u_1 = u_2\}.
    \end{align*} 

\begin{definition}[$k$-pointed tube]
    Let $k\geq 1$ be fixed. We construct a $3$-polytope $\mathcal{Q}_k$, called the \emph{$k$-pointed tube}, by constructing its graph as follows: Let $P_k$ be the path on $k$ vertices with $V(P_k) = \{0, 1, \dotsc, k-1\}$. Let $C$ be the cycle on three vertices with $V(C) = \{0, 1, 2\}$. Start by writing $G = C\times P_k$. Now, write $V(\mathcal{Q}_k) = V(G) \cup \{x, y\}$ where $x, y$ are separately labeled vertices. Then, write
    \begin{align*}
        E(\mathcal{Q}_k) &= E(G) \cup \bigcup_{i=0}^2 (\{x, (i, 0)\} \cup \{y, (i, k-1)\} ).
    \end{align*}
\end{definition}

Two copies of the $k$-pointed tube for $k=3, 5$ are illustrated in~\Cref{subfig:resis-pencil}. 

\begin{example}\label{thm:pencil-positive}
    The $3$-polytope $\mathcal{Q}_k$ has everywhere positive resistance curvature for each $k\geq 1$.
\end{example}

The proof of~\Cref{thm:pencil-positive} is long and requires a systematic derivation of the effective resistances between edges in $\mathcal{Q}_k$, which in turn requires a full spectral decomposition of its Laplacian. We include a proof of the everywhere positivity of this example family in~\cref{sec:appendix}.

\subsection{Conditions for negative resistance curvature}\label{subsec:resistance-negative}

In this subsection we obtain a degree-based lower bound for the effective resistance between vertices in a graph and use it to establish criteria for the existence of a negative curvature vertex. 

\begin{theorem}\label{thm:resistance-subset-lower-bound}
    Let $G=(V,E)$ be any graph, let $A\subseteq V$ be nonempty, and $u, v\in A$ fixed. Let $\mathbf{L}_A$ denote the principal submatrix of the Laplacian matrix $\mathbf{L}$ with rows and columns indexed by vertices in $A$. Then we have
        \begin{align}\label{eq:resistance-lower-bound}
            r_{uv} \geq (\mathbf{1}_u-\mathbf{1}_v)^\top \mathbf{L}_A^\dagger (\mathbf{1}_u - \mathbf{1}_v).
        \end{align}
\end{theorem}

We note that if $A\subseteq V$ does not contain an entire connected component of $G$, then $\mathbf{L}_A$ will in fact be invertible (since it is strictly diagonally dominant in at least one row or column) and in turn $\mathbf{L}_A^\dagger = \mathbf{L}_A^{-1}$.

\begin{proof}[Proof of~\Cref{thm:resistance-subset-lower-bound}]
    Using~\cref{prop:flow-resistance}, we have
        \begin{align*}
            r_{uv} &= \inf\left\{\sum_{e\in E'}|\mathbf{J}_e|^2 \;:\; \mathbf{J}\in\mathbb{R}^{E'}, \mathbf{B}\mathbf{J} = \mathbf{1}_u-\mathbf{1}_v \right\}.
        \end{align*}
    We can first apply a relaxation of the constraint $\mathbf{B}\mathbf{J} = \mathbf{1}_u-\mathbf{1}_v$ by considering $\mathbf{J}$ such that $(\mathbf{B}\mathbf{J})(x) = ( \mathbf{1}_u-\mathbf{1}_v )(x)$ for $x\in A$. This leads to the inequality
        \begin{align}\label{eq:relaxed-infimum}
            r_{uv} &\geq \inf\left\{\sum_{e\in E'}|\mathbf{J}_e|^2  \;:\;\mathbf{J}\in\mathbb{R}^{E'},\;(\mathbf{B}\mathbf{J})(x) = ( \mathbf{1}_u-\mathbf{1}_v )(x)\;\;\forall\,x\in A \right\}.
        \end{align}
    Now we claim the infimum in~\Cref{eq:relaxed-infimum} is realized by the right-hand side in~\cref{eq:resistance-lower-bound}. To see this, let $\mathbf{B}_A\in\mathbb{R}^{|A|\times |E|}$ be the submatrix of $\mathbf{B}$ with rows indexed by the vertices in $A$ and columns unchanged. Then the relaxed constraint $(\mathbf{B}\mathbf{J})(x) = ( \mathbf{1}_u-\mathbf{1}_v )(x)$ for $x\in A$ can be recast as $\mathbf{B}_A\mathbf{J}= \mathbf{1}_u-\mathbf{1}_v$, where with a slight abuse of notation, we identify $\mathbf{1}_u-\mathbf{1}_v\in\mathbb{R}^{|A|}$ with its restriction to the vertices in $A$. We then have, upon inspection and the basic properties of the matrix pseudoinverse, that
        \begin{align*}
            \inf\left\{\sum_{e\in E'}|\mathbf{J}_e|^2  \;:\;\mathbf{J}\in\mathbb{R}^{E'},\;\mathbf{B}_A\mathbf{J}= \mathbf{1}_u-\mathbf{1}_v\right\} &= \| \mathbf{B}_A^\dagger(\mathbf{1}_u-\mathbf{1}_v)\|_2^2\\
            &= (\mathbf{1}_u-\mathbf{1}_v)^\top (\mathbf{B}_A^\dagger)^\top \mathbf{B}_A^\dagger(\mathbf{1}_u-\mathbf{1}_v)\\
            &= (\mathbf{1}_u-\mathbf{1}_v)^\top \mathbf{L}_A^\dagger (\mathbf{1}_u-\mathbf{1}_v)
        \end{align*}
    since $\mathbf{B}_A \mathbf{B}_A^\top = \mathbf{L}_A$. The claim follows.
\end{proof}

\Cref{thm:resistance-subset-lower-bound} can be used to obtain a degree-based lower bound on the effective resistance between adjacent vertices in a graph.

\begin{theorem}\label{thm:resistance-degree-bound}
    Let $G=(V, E)$ be a graph and fix adjacent vertices $u, v\in V$. Assume for simplicity that either $d_u\geq 2$ or $d_v\geq 2$ (i.e., that $\{u, v\}$ is not a connected component of $G$). Then it holds
        \begin{align*}
            r_{uv} \geq \max\left\{\frac{d_v + d_u-2}{d_ud_v - 1}, \frac{4}{d_u + d_v + 2}\right\}.
        \end{align*}
\end{theorem}

\begin{proof}
    Let $A=\{u, v\}\subseteq V$ and let $\mathbf{b}=\mathbf{1}_u - \mathbf{1}_v$. Then we have
        \begin{align*}
            \mathbf{L}_A = \begin{bmatrix}
                d_u & -1\\
                -1 & d_v
            \end{bmatrix},\,\, \mathbf{L}_A^{-1} = \frac{1}{d_ud_v - 1}\begin{bmatrix}
                d_v & 1\\
                1 & d_u
            \end{bmatrix}.
        \end{align*}    
    In turn,
        \begin{align*}
            \mathbf{b}^\top  \mathbf{L}_A^{-1} \mathbf{b} &= \frac{1}{d_ud_v - 1}\mathbf{b}^\top\begin{bmatrix}
                d_v - 1\\
                1 - d_u
            \end{bmatrix} = \frac{d_v + d_u-2}{d_ud_v - 1}.
        \end{align*}
    The first claim then follows by Theorem~\cref{thm:resistance-subset-lower-bound}. On the other hand, since $\mathbf{L}_A$ is positive definite (having assumed at least one of the vertices $u, v$ is not degree one, $A$ does not contain an entire connected component), we must have, by the Cauchy-Schwarz inequality with and $\mathbf{x}\in\mathbb{R}^{|A|}$ fixed with $\mathbf{x}\neq\mathbf{0}$,
        \begin{align*}
            |\mathbf{x}^\top \mathbf{b}|^2 &= |(\mathbf{L}_A^{1/2} \mathbf{x})^\top (\mathbf{L}_A^{-1/2} \mathbf{b})|^2 \leq \|\mathbf{L}_A^{1/2}\mathbf{x}\|^2\|\mathbf{L}_A^{-1/2}\mathbf{b}\|^2,
        \end{align*}
    or, $\mathbf{b}^\top \mathbf{L}_A^{-1} \mathbf{b} \geq \frac{|\mathbf{x}^\top \mathbf{b}|^2}{\mathbf{x}^\top \mathbf{L}_A\mathbf{x}}$. Since $\mathbf{x}$ was arbitary, we can take for example $\mathbf{x}=\mathbf{b}$, and obtain
        \begin{align*}
            \mathbf{b}^\top \mathbf{L}_A ^{-1} \mathbf{b} \geq \frac{\|\mathbf{b}\|^4}{\mathbf{b}^\top \mathbf{L}_A \mathbf{b}}\geq \frac{4}{(d_u+1)-(-1-d_v)}.
        \end{align*}
    The theorem follows.
\end{proof}

~\Cref{thm:resistance-degree-bound}, in turn, leads to a degree-based criterion for the existence of a vertex in a graph with negative resistance curvature, as follows.

\begin{corollary}\label{cor:resistance-curv-lowerbound}
    Let $G=(V,E)$ be any graph, and suppose $v\in V$ satisfies the following two conditions:
        \begin{enumerate}[label={\normalfont (\roman*)}]
            \item $d_v\geq 2$, and
            \item For each $u\sim v$, $d_u\leq d_v -2$.
        \end{enumerate}
    Then the resistance curvature $\resistance{v}$ at vertex $v$ satisfies $\resistance{v}\leq 0$.
\end{corollary}

\begin{proof}
    From~\Cref{thm:resistance-degree-bound}, we have that for each $u\sim v$, the resistance $r_{uv}$ satisfies
        \begin{align*}
            r_{uv} \geq \frac{4}{d_u + d_v + 2} \geq \frac{2}{d_v},
        \end{align*}
    therefore,
        \begin{align*}
            \resistance{v} &= 1-\frac{1}{2}\sum_{u\sim v}r_{uv} \leq 1-\frac{1}{2}\sum_{u\sim v}\frac{2}{d_v} \leq 0.
        \end{align*}
\end{proof}

From~\Cref{cor:resistance-curv-lowerbound} we may deduce that, for example, pyramids have negative curvature at the apex whenever the base polygon contains five or more vertices. Moreover~\Cref{cor:resistance-curv-lowerbound} suggests that the resistance positive \emph{graphs} are limited to those which are, in a suitably weak sense, close to being degree regular. We remark finally that the techniques utilized in the proof of~\Cref{thm:resistance-degree-bound} could conceivably be extended in more sophisticated ways utilizing higher-order information from the $1$-skeleton of a polytope, and this direction is promising for future work. 

We finish with a conjecture about resistance curvature and simple 3-dimensional polytopes, that would imply the scarcity of the most relevant resistance positive polytopes in dimension $3$. If the vast majority of faces have at least six sides, their incidences may guarantee induced subgraphs that may be combined with~\cref{thm:resistance-subset-lower-bound} to guarantee the negativity. Eberhard's theorem implies that several large incidences make it plausible that the following is true:

\begin{conjecture}
    There are finitely many simple $3$-dimensional polytopes that are resistance positive and are not isomorphic to a prism over a polygon. 
\end{conjecture}

\section*{Closing Remarks and Acknowledgements}

In this paper we investigated Forman--Ricci and Resistance curvatures of convex polytopes. We remark that there are many other notions of curvature we did not discuss here, including those of Ollivier~\cite{Ollivier2009}, Lin-Lu-Yau~\cite{lin2011ricci}, or Steinerberger~\cite{steinerberger2023curvature}, which have been studied at length in the literature. Similarly, we are aware of only a few papers trying to compare curvatures~\cite{Samal2018,WatanabeYamada2018,robertson2024discrete}. We believe that exploring the variants of our results for other curvatures could be of interest. Our experiments suggest that the same kind of results will hold. 

\vskip 12pt

The authors thank Stefan Steinerberger for several useful comments and suggestions. The first two authors are grateful to NSF for support through grants DMS-2348578 and DMS-2434665. The fourth author is partially supported by ANID FONDECYT Iniciaci\'on grant \#11221076.

\appendix

\section{Computational Details for Example 3.7}\label{sec:appendix}

This appendix contains a detailed derivation of the claim made in~\cref{thm:pencil-positive} of the main text. We begin by restating the definition of the $3$-polytope termed the $k$-pointed tube below.

\begin{definition}[$k$-pointed tube]
    We construct a $3$-polytope $\mathcal{Q}_k$, called the \emph{$k$-pointed tube}, by constructing its graph as follows: Let $P_k$ be the path on $k$ vertices with $V(P_k) = \{0, 1, \dotsc, k-1\}$. Let $C_3$ be the cycle on three vertices with $V(C_3) = \{0, 1, 2\}$. Start by writing $G = C_3\times P_k$. Now, write $V(\mathcal{Q}_k) = V(G) \cup \{x, y\}$ where $x, y$ are separately labeled vertices. Then, write
    \begin{align*}
        E(\mathcal{Q}_k) &= E(G) \cup \bigcup_{i=0}^2 (\{x, (i, 0)\} \cup \{y, (i, k-1)\} ).
    \end{align*}
\end{definition}

Here, the step $G = C_3\times P_k$ refers to the Cartesian product of graphs operation introduced in the main text. The claim presented in~\cref{thm:pencil-positive} is proved at the end of this appendix in~\cref{thm:positivity-tube}. To establish this we establish several computational lemmas. We denote by $G_k$ the graph of $\mathcal{Q}_k$. The main task that needs to be completed is to compute closed-form expressions for the effective resistances of edges in $G_k$. The edges of $G_k$ can be partitioned into three sets, as follows:
    \medskip
    \begin{center}
        \begin{tabular}{c|c}
            ``cycle edges'' & $\big\{\{(i, j), (i+1, j)\} : 0\leq j\leq k-1, 0\leq i\leq 2\big\}$ \\\hline
            ``path edges'' & $\big\{\{(i, j), (i, j + 1)\} : 0\leq j\leq k-2, 0\leq i\leq 2\big\}$ \\\hline
            ``cap edges'' & $\big\{\{(i, 0), x\}, \{(i, k-1), y\} : 0\leq i\leq 2\big\}$
        \end{tabular}
    \end{center}
    \medskip
Note that addition on the first vertex coordinate is carried out modulo $3$. By exploiting vertex symmetries in the graph, it is straightforward to deduce that the edge effective resistance in $G_k$ can be categorized as the following family of numbers:
    \begin{align*}
        r_{\text{cycle}}( j ; k) &= r_{(i, j), ({i+1}, j)},\text{ and which does not depend on }0\leq i\leq 2,\\
        r_{\text{path}}( j ; k) &= r_{(i, j), ({i}, j + 1)},\text{ and which does not depend on }0\leq i\leq 2,\\
        r_{\text{cap}}(k) &= r_{(i, 0), x} = r_{(i, k-1), y},\text{ and which does not depend on }0\leq i\leq 2.
    \end{align*}
Here, $k\geq 1$ and $0\leq j\leq k-2$. The first lemma establishes a block diagonalization of the corresponding Laplacian matrix. As a matter of notation, if $S$ is any set, we write $\ell(S)$ to denote the linear space of functions $f:S\rightarrow\mathbb{R}$.

\begin{lemma}\label{lem:block-diag}
    The Laplacian matrix $\mathbf{L}(G_k)$ admits a block diagonalization of the form
        \begin{align}\label{eq:lap-block-repn}
            \mathbf{L}(G_k) = \widetilde{\mathbf{U}} \begin{pmatrix}
            (\widetilde{\Delta}_k^{(0)}) & & \\
            & (\Delta_k^{(3)}) & \\
            & & (\Delta_k^{(3)})
        \end{pmatrix} \widetilde{\mathbf{U}}^\top,
        \end{align}
    where
        \begin{align*}
            \widetilde{\Delta}_k^{(0)}=
            \begin{pmatrix}
            3 & -\sqrt3 &        &        &        &   \\
            -\sqrt3 & 2 & -1     &        &        &    \\
                    & -1 & 2 & -1&        &        \\
                    &    & \ddots & \ddots & \ddots &   \\
                    &    &        & -1 & 2 & -\sqrt3 \\
                   &    &        &    & -\sqrt3 & 3
            \end{pmatrix}\in\mathbb{R}^{(k+2)\times (k+2)}, 
        \end{align*}
        \begin{align*}
            \Delta_k^{(3)} &= \begin{pmatrix}
            5 & -1 &        &        & \\
            -1& 5  & -1     &        &  \\
               &-1 & 5      & \ddots &  \\
               &   & \ddots & \ddots & -1\\
              &   &        & -1 & 5
            \end{pmatrix}\in\mathbb{R}^{k\times k},
        \end{align*}
    and 
        \begin{align*}
            \widetilde{\mathbf{U}} &= \begin{pmatrix}
                1 & 0\\
                0 & 1
            \end{pmatrix} \oplus \bigoplus_{j=1}^{k} \begin{pmatrix}
                \tfrac{1}{\sqrt{3}} & \tfrac{1}{\sqrt{2}}&\tfrac{1}{\sqrt{6}}\\
                \tfrac{1}{\sqrt{3}} & -\tfrac{1}{\sqrt{2}}& \tfrac{1}{\sqrt{6}}\\
                \tfrac{1}{\sqrt{3}} & 0 & -\tfrac{2}{\sqrt{6}}\\
            \end{pmatrix}.
        \end{align*}
\end{lemma}

\begin{proof}
    We begin by setting, for $0\leq i\leq 2$, the vectors $\mathbf{u}_i\in\mathbb{R}^3$ given by
        \begin{align*}
            \mathbf{u}_0&=\tfrac1{\sqrt3}(1,1,1),&
            \mathbf{u}_1&=\tfrac1{\sqrt2}(1,-1,0),&
            \mathbf{u}_2&=\tfrac1{\sqrt6}(1,1,-2).
        \end{align*}
    These vectors form an orthonormal basis for $\mathbb{R}^3$ and it is straightforward to verify that the cycle Laplacian matrix $\mathbf{L}(C_3)$ satisfies
        \begin{align*}
            \mathbf{L}({C_3})\mathbf{u}_0&=0,& 
            \mathbf{L}({C_3})\mathbf{u}_1&=3\mathbf{u}_1,&
            \mathbf{L}({C_3})\mathbf{u}_2&=3\mathbf{u}_2.
        \end{align*}
    Any function $f:V(C_3)\times V(P_k)\to\mathbb R$ therefore decomposes uniquely as
        \begin{align}\label{eq:expansion-fourier}
        f(i,j)=
            f^{(0)}(j)\mathbf{u}_0(i)+f^{(1)}(j)\mathbf{u}_1(i)+f^{(2)}(j)\mathbf{u}_2(i),
        \end{align}
    where the scalar coefficients $f^{(s)}(j)$, $s=0,1,2$, depend only on
    the index $j$. If we identify the linear space $\ell(V(C_3\times P_k))$ with $\mathbb{R}^{3\times k}$,~\cref{eq:expansion-fourier} can be expressed as
        \begin{align*}
            \begin{pmatrix}
            f(0,0) & f(0,1) & \cdots & f(0,k-1)\\
            f(1,0) & f(1,1) & \cdots & f(1,k-1)\\
            f(2,0) & f(2,1) & \cdots & f(2,k-1)
            \end{pmatrix}
            &=
            U
            \begin{pmatrix}
            f^{(0)}(0) & f^{(0)}(1) & \cdots & f^{(0)}(k-1)\\
            f^{(1)}(0) & f^{(1)}(1) & \cdots & f^{(1)}(k-1)\\
            f^{(2)}(0) & f^{(2)}(1) & \cdots & f^{(2)}(k-1)
            \end{pmatrix}
        \end{align*}
    with $U = \begin{pmatrix}
        \mathbf{u}_0 & \mathbf{u}_1 & \mathbf{u}_2
    \end{pmatrix}\in\mathbb{R}^{3\times 3}$. Let $\mathcal L$ denote the Laplacian of $C_3\times P_k$ {before} the caps $x,y$ are attached. For fixed $0\leq i \leq 2$ and $0\leq j\leq k-1$ the neighbors of $(i,j)$ are
        \begin{align*}
            (i\pm 1,j),\quad (i,j-1),\quad (i,j+1)
        \end{align*}
    with the obvious adjustments when $j = 0, k-1$. Therefore if $f\in \ell(V(C_3\times P_k))$, one has
        {\footnotesize\begin{align*}
            (\mathcal L f)(i,j)= \begin{cases}
                3f(i,j)-\bigl[f(i+1,j)+f(i-1,j)\bigr] - f(i,j+1) &\text{ if } j = 0,\\
                4f(i,j)-\bigl[f(i+1,j)+f(i-1,j)\bigr] -\bigl[f(i,j-1)+f(i,j+1)\bigr] &\text{ if } 0 < j < k-1,\\
                3f(i,j)-\bigl[f(i+1,j)+f(i-1,j)\bigr] - f(i,j-1) &\text{ if } j = k-1.\\
            \end{cases}
        \end{align*}}
    Inserting the expansion set up in~\cref{eq:expansion-fourier}, we have that for $0\leq i\leq 2$ and $0 < j < k-1$ fixed:
        \begin{align*}
            (\mathcal{L}f)(i, j) &= \sum_{s=0}^2 4f^{(s)}(j)\mathbf{u}_s(i)-\bigl[  f^{(s)}(j)\mathbf{u}_s(i+1) + f^{(s)}(j)\mathbf{u}_s(i-1) \bigr]\\
            &\qquad\qquad-\bigl[ f^{(s)}(j-1)\mathbf{u}_s(i) +f^{(s)}(j+1)\mathbf{u}_s(i)\bigr]\\
            &= \sum_{s=0}^2 f^{(s)}(j)(\mathbf{L}(C_3)\mathbf{u}_s)(i) + \mathbf{u}_s(i)(2f^{(s)}(j) -f^{(s)}(j-1) -f^{(s)}(j+1)) \\
            &=\sum_{s=0}^2 \mathbf{u}_s(i)\left[(2+\lambda_s)f^{(s)}(j) -f^{(s)}(j-1) -f^{(s)}(j+1)\right]
        \end{align*}
    where $\lambda_0 = 0$ and $\lambda_1 = \lambda_2= 3$. Similarly, if $j=0, k-1$, we have
        \begin{align*}
            (\mathcal{L}f)(i, j) &= \sum_{s=0}^2 \mathbf{u}_s(i)\left[(2+\lambda_s)f^{(s)}(j) -f^{(s)}(j+1)\right],\quad j=0\\
            (\mathcal{L}f)(i, j) &= \sum_{s=0}^2 \mathbf{u}_s(i)\left[(2+\lambda_s)f^{(s)}(j) -f^{(s)}(j-1)\right],\quad j=k-1.
        \end{align*}
    Inspecting the expressions above, we define the ``longitudinal'' operators $\Delta_k^{(0)}, \Delta_k^{(3)}:\ell(V(P_k))\rightarrow \ell(V(P_k))$ given by their action on functions $g\in \ell(V(P_k))$ as follows:
        \begin{align*}
            (\Delta_k^{(0)}g)(j) &= \begin{cases}
                2g(j)-g(j+1) & \text{ if }j=0,\\
                -g(j-1)+2g(j)-g(j+1) & \text{ if }0<j<k-1,\\
                -g(j-1) + 2g(j) & \text{ if }j=k-1,\\
            \end{cases}\quad g\in\ell(V(P_k)),
        \end{align*}
    and
        \begin{align}\label{eq:anti-block}
            (\Delta_k^{(3)}g)(j) &= \begin{cases}
                5g(j)-g(j+1) & \text{ if }j=0,\\
                -g(j-1)+5g(j)-g(j+1) & \text{ if }0<j<k-1,\\
                -g(j-1) + 5g(j) & \text{ if }j=k-1,\\
            \end{cases}\quad g\in\ell(V(P_k)).
        \end{align}
    We call $\Delta_k^{(0)}$ the {symmetric block} ($\lambda_0=0$), and
    $\Delta_k^{(3)}$ the {antisymmetric block}, and identify the operators with their matrix representations under the standard basis. By~\cref{eq:expansion-fourier} and the preceding, we have that the Laplacian $\mathbf{L}(C_3\times P_k)$ admits the decomposition
        \begin{align*}
            \mathbf{L}(C_3\times P_k) = \mathbf{U} \begin{pmatrix}
            \Delta_k^{(0)} & & \\
            & \Delta_k^{(3)} & \\
            & & \Delta_k^{(3)}
        \end{pmatrix} \mathbf{U}^\top,\qquad \mathbf{U} = \mathrm{diag}(\underbrace{U, U, \dotsc,  U}_{k\text{ times}}).
        \end{align*}
    Now suppose we add to $C_3\times P_k$ the two additional ``cap'' vertices $x, y$ and thereby obtain the graph $G_k$. The space $\ell(V(G_k))$ can be identified as $\mathbb{R}\oplus \mathbb{R} \oplus \ell(V(C_3\times P_k))$. If we set 
        \begin{align*}
            \widetilde{\mathbf{U}} = \mathrm{diag}(I_2, \underbrace{U, U, \dotsc,  U}_{k\text{ times}}),
        \end{align*}
    then we have
        \begin{align}\label{eq:block-decomp-L}
            \mathbf{L}(G_k) = \widetilde{\mathbf{U}} \begin{pmatrix} 
            \mathbf{A}_1 & \mathbf{A}_2 & \mathbf{A}_3 & \mathbf{A}_3\\
            \mathbf{A}_2^\top &\Delta_k^{(0)} & & \\
            \mathbf{A}_3^\top && \Delta_k^{(3)} & \\
            \mathbf{A}_3^\top && & \Delta_k^{(3)}
        \end{pmatrix} \widetilde{\mathbf{U}}^\top
        \end{align}
    where $\mathbf{A}_1\in\mathbb{R}^{2\times 2}$, $\mathbf{A}_2\in\mathbb{R}^{2\times k}$, $\mathbf{A}_3\in\mathbb{R}^{2\times k}$ are to be determined. We claim $\mathbf{A}_3=0$. Let $0\leq j \leq k-1$ be fixed. Let $(\mathbf{A}_3)_{x, j}$ denote the entry of $\mathbf{A}_3$ in the first row (indexed by $x$) and the $j$-th column. We write
        \begin{align*}
            (\mathbf{A}_3)_{x, j} = \mathbf{1}_x^\top \widetilde{\mathbf{U}}^\top \mathbf{L}(G_k)\widetilde{\mathbf{U}} \mathbf{1}_{k+2+j}
        \end{align*}
    Here, $\mathbf{1}_x$ is the indicator vector of the first coordinate (indexed by $x$) and $\mathbf{1}_{k+2+j}$ is the indicator vector of the $(k+2+j)$-th coordinate, which is chosen so as to capture the entry of $\mathbf{A}_3$ as it appears in the first row of the block decomposition of $\mathbf{L}(G_k)$ in~\cref{eq:block-decomp-L}. Note that
        \begin{align*}
            (\widetilde{\mathbf{U}}\mathbf{1}_x)^\top \mathbf{L}(G_k) &= \mathbf{1}_x^\top \mathbf{L}(G_k)\\
            &= 3\mathbf{1}_x^\top - \mathbf{1}_{(0, 0)}^\top - \mathbf{1}_{(1, 0)}^\top - \mathbf{1}_{(2, 0)}^\top
        \end{align*}
    And therefore
        \begin{align*}
            (\mathbf{A}_3)_{x,j}
            &=(3\mathbf{1}_x^\top - \mathbf{1}_{(0,0)}^\top - \mathbf{1}_{(1,0)}^\top - \mathbf{1}_{(2,0)}^\top)\,\widetilde{\mathbf{U}}\,\mathbf{1}_{k+2+j}\\
            &=-\sum_{i=0}^2 \mathbf{1}_{(i,0)}^\top\,\widetilde{\mathbf{U}}\,\mathbf{1}_{k+2+j}.
        \end{align*}
    If the nonzero entries of $\widetilde{\mathbf{U}}\,\mathbf{1}_{k+2+j}$ corresponds to $\mathbf{u}_s$ for $s=1,2$ at level $j=0$, then
        \begin{align*}
            \mathbf{1}_{(i,0)}^\top\,\widetilde{\mathbf{U}}\,\mathbf{1}_{k+2+j}
            = \mathbf{u}_s(i),
            \quad
            \sum_{i=0}^2 \mathbf{u}_s(i)=0,
        \end{align*}
    otherwise the inner product vanishes. In either case the sum is zero, so $(\mathbf{A}_3)_{x,j}=0$.  A similar calculation for the row indexed by $y$ shows $(\mathbf{A}_3)_{y,j}=0$ for all $j$.  Therefore $\mathbf{A}_3=0$ and we have a block diagonalization of $\mathbf{L}(G_k)$ in the form
        \begin{align*}
            \mathbf{L}(G_k) = \widetilde{\mathbf{U}} \begin{pmatrix}
            \widetilde{\Delta}_k^{(0)} & & \\
            & \Delta_k^{(3)} & \\
            & & \Delta_k^{(3)}
        \end{pmatrix} \widetilde{\mathbf{U}}^\top, \qquad \widetilde{\Delta}_k^{(0)} = P\begin{pmatrix}
            \mathbf{A}_1 & \mathbf{A}_2 \\
            \mathbf{A}_2 ^\top & \Delta_k^{(0)}
        \end{pmatrix}P^\top,
        \end{align*}
    where $P$ is the permutation matrix which shifts the coordinate indexing the vertex $y$ to the last slot of the $(k+2)\times (k+2)$ matrix. From~\cref{eq:anti-block}, it follows that
        \begin{align*}
            \Delta_k^{(3)} &= \begin{pmatrix}
            5 & -1 &        &        & \\
            -1& 5  & -1     &        &  \\
            &-1 & 5      & \ddots &  \\
            &   & \ddots & \ddots & -1\\
            &   &        & -1 & 5
            \end{pmatrix}
        \end{align*}
    On the other hand, recall that $\widetilde{\Delta}_k^{(0)} \in\mathbb R^{(k+2)\times(k+2)}$ is obtained from the block
        \begin{align*}
            \begin{pmatrix}
                \mathbf{A}_1 & \mathbf{A}_2\\[2pt]
                \mathbf{A}_2^{\top} & \Delta_k^{(0)}
                \end{pmatrix}
        \end{align*}
    after the permutation $P$ that sends the coordinate indexed by the vertex $y$ to the last position. It therefore suffices to identify the matrices $\mathbf{A}_1$ and $\mathbf{A}_2$ explicitly. For the caps $x$ and $y$ one has
        \begin{align*}
            \mathbf{L}(G_k)\mathbf{1}_x=3\mathbf{1}_x-\sum_{i=0}^2\mathbf{1}_{(i,0)},
            \qquad
            \mathbf{L}(G_k)\mathbf{1}_y=3\mathbf{1}_y-\sum_{i=0}^2\mathbf{1}_{(i,k-1)},
        \end{align*}
    so that, in the $\{x,y\}$-coordinates,
        \begin{align*}
            \mathbf{A}_1=\begin{pmatrix}3&0\\[2pt]0&3\end{pmatrix}.
        \end{align*}
    Now let $0\le j\le k-1$ be fixed and let $\mathbf{v}^{(0)}_j:=\mathbf{U}\mathbf{1}_{3j}$ contain a copy of $\mathbf{u}_0$ at level $j$. Then it holds
        \begin{align*} 
            \mathbf{1}_{(i,0)}^{\top}\mathbf{v}^{(0)}_j=
            \begin{cases}
            \frac1{\sqrt3},&j=0,\\
            0,&j\ne 0,
            \end{cases}
            \qquad
            \mathbf{1}_{(i,k-1)}^{\top}\mathbf{v}^{(0)}_j=
            \begin{cases}
            \frac1{\sqrt3},&j=k-1,\\
            0,&j\ne k-1.
            \end{cases}
        \end{align*}
    Using the expression for $\mathbf{L}(G_k)\mathbf{1}_x$ above, we obtain, 
        \begin{align*}
            (\mathbf{A}_2)_{xj}=\mathbf{1}_x^{\top}\mathbf{L}(G_k)\mathbf{v}^{(0)}_j
            =-\,\sum_{i=0}^2\mathbf{1}_{(i,0)}^{\top}\mathbf{v}^{(0)}_j
            =-\sqrt3\,\delta_{j,0}.
        \end{align*}
    An identical calculation with $\mathbf{1}_y$ gives
        \begin{align*}
            (\mathbf{A}_2)_{yj}=-\sqrt3\,\delta_{j,k-1}.
        \end{align*}
    Hence
        \begin{align*}
            \mathbf{A}_2=\begin{pmatrix}
                -\sqrt3 & 0 &\cdots &0\\
                0&\cdots &0 &-\sqrt3
                \end{pmatrix},
        \end{align*}
    where the first (resp.\ last) column corresponds to $j=0$
    (resp.\ $j=k-1$). Inserting $\mathbf{A}_1$, $\mathbf{A}_2$, and
        \begin{align*}
            \Delta_k^{(0)}=\begin{pmatrix}
                2&-1\\[2pt]-1&\ddots&\ddots\\
                &\ddots&2&-1\\
                &&-1&2
                \end{pmatrix}
        \end{align*}
    into $\begin{pmatrix}\mathbf{A}_1&\mathbf{A}_2\\[2pt]\mathbf{A}_2^{\top}&\Delta_k^{(0)}\end{pmatrix}$, and then apply the permutation $P$ which moves the row/column indexed by $y$ to the bottom-right corner.  The result is the tridiagonal matrix
        \begin{align*}
            \widetilde{\Delta}_k^{(0)}=
            \begin{pmatrix}
            3 & -\sqrt3 &        &        &        & 0  \\
            -\sqrt3 & 2 & -1     &        &        &    \\
                    & -1 & 2 & -1&        &        \\
                    &    & \ddots & \ddots & \ddots &   \\
                    &    &        & -1 & 2 & -\sqrt3 \\
            0       &    &        &    & -\sqrt3 & 3
            \end{pmatrix},
        \end{align*}
    The claim follows.
\end{proof}

The next lemma provides closed-form expressions for the Moore-Penrose inverses of the components of the block matrices described above.

\begin{lemma}\label{lem:block-pinv}
    Assume the notation and conventions of~\cref{lem:block-diag}. Let $\varphi = \mathrm{arccosh}\bigl(\tfrac{5}{2}\bigr)$. Then
        \begin{align*}
            (\Delta_k^{(3)})^{-1}_{ij}
            &=
            \begin{cases}
            \displaystyle
            \frac{\mathrm{sinh}(i\varphi)\,\mathrm{sinh}\bigl((k-j+1)\varphi\bigr)}
                {\mathrm{sinh}(\varphi)\,\mathrm{sinh}\bigl((k+1)\varphi\bigr)},
            & i \le j,\\
            \displaystyle
            \frac{\mathrm{sinh}(j\varphi)\,\mathrm{sinh}\bigl((k-i+1)\varphi\bigr)}
                {\mathrm{sinh}(\varphi)\,\mathrm{sinh}\bigl((k+1)\varphi\bigr)},
            & i > j,
            \end{cases}
        \end{align*}
    and 
        \begin{align*}
            (\widetilde{\Delta}_k^{(0)})^\dagger_{ij} = (\widetilde{\Delta}_k^{(0)})^\dagger_{ji} &= \frac{\omega_i \omega_j}{3}\left(h_k(i-1) + h_k(k+2-j)- c_k(i, j)\right),\quad 1\leq i\leq j\leq k+2,
        \end{align*}    
    where
        \begin{align*}
            (\omega_1, \omega_2, \dotsc, \omega_{k+1}, \omega_{k+2}) &:= \frac{1}{\sqrt{3k+2}} (1, \sqrt{3}, \sqrt{3}, \dotsc, \sqrt{3}, 1),\\
            h_k(x) &= \frac{x(6x^2 - 3x - 1)}{2(3k+2)}, \quad x\in\mathbb{R}\\
            c_k(i, j) &:= \frac{(j-i)}{2(3k+2)}\bigl(2(3k+4)-3(i+j-1)\bigr), \quad i, j\in\mathbb{Z}.
        \end{align*}
\end{lemma}

Note that by~\cref{lem:block-diag}, it follows that
    \begin{align}\label{eq:lap-pinv}
        \mathbf{L}(G_k)^\dagger = \widetilde{\mathbf{U}} \begin{pmatrix}
            (\widetilde{\Delta}_k^{(0)})^\dagger & & \\
            & (\Delta_k^{(3)})^{-1} & \\
            & & (\Delta_k^{(3)})^{-1}
        \end{pmatrix} \widetilde{\mathbf{U}}^\top.
    \end{align}

Before proving~\cref{lem:block-pinv}, we recall a useful theorem for inverting symmetric tridiagonal matrices below.

\begin{lemma}[Bendito, Carmona, Encinas~\cite{bendito2010generalized}]\label{lem:inverting-tridiagonal}
    Let $M\in\mathbb{R}^{n\times n}$ be the tridiagonal symmetric matrix
        \begin{align*}
            M = \begin{pmatrix}
                d_1 & -c_1 & & & & \\
                -c_1 & d_2 & -c_2 & & & \\
                & -c_2 & d_3 & -c_3 & & \\
                & &\ddots &\ddots &\ddots &\\
                & & &-c_{n-2}&d_{n-1}&-c_{n-1}\\
                & & & &-c_{n-1}&d_n\\
            \end{pmatrix},
        \end{align*}
    where $c_i \geq 0$ for $1\leq i \leq n-1$ and $d_i \geq 0$ for $1\leq i \leq n$. Assume there exist $\omega_1, \omega_2, \dotsc,\omega_n >0$ with $\sum_{i=1}^n\omega_i^2 = 1$ such that
        \begin{align}\label{eq:omega-js}
            d_j = \frac{1}{\omega_j}\left(c_j\omega_{j+1} + c_{j-1}\omega_{j-1}\right),\quad 1\leq j\leq n, \quad c_0 = c_n = \omega_0 = \omega_{n+1} := 0.
        \end{align}
    Then the Moore-Penrose inverse $M^\dagger$ has entries given by
        {\footnotesize\begin{align*}
            (M^\dagger)_{ij} =  (M^\dagger)_{ji} &= \omega_i \omega_j \left[
                \sum_{k=1}^{i-1} 
                \frac{\left( \sum_{\ell=1}^{k} \omega_\ell^2 \right)^2}{c_k \omega_k \omega_{k+1}}
                +
                \sum_{k=i}^{n-1}
                \frac{\left( \sum_{\ell=k+1}^{n} \omega_\ell^2 \right)^2}{c_k \omega_k \omega_{k+1}}
                -
                \sum_{k=i}^{j-1}
                \frac{\left( \sum_{\ell=k+1}^{n} \omega_\ell^2 \right)}{c_k \omega_k \omega_{k+1}}
                \right],
        \end{align*}}
    where $1\leq i\leq j\leq n$.
\end{lemma}

\begin{proof}
    Since $\Delta_k^{(3)}$ is real symmetric and strictly diagonally dominant it is positive-definite and thus invertible. Hence its Moore-Penrose inverse coincides with its ordinary inverse. A standard method for tridiagonal matrices is to solve a difference equation in lieu of Gaussian elimination. We fix $1\le\ell\le k$ and solve $\Delta_k^{(3)}\mathbf{x}=\mathbf{1}_\ell$. Writing $x_0=x_{k+1}=0$, the components satisfy
        \begin{align*}
            5x_i-x_{i-1}-x_{i+1}=\delta_{i\ell},
            \qquad 1\le i\le k .
        \end{align*}
    For $i\neq\ell$ the homogeneous recurrence $5x_i-x_{i-1}-x_{i+1}=0$ has characteristic polynomial $r^{2}-5r+1=0$ with distinct roots 
        \begin{align*}
            r_1 &= \frac{5 + \sqrt{21}}{2} = e^\varphi, \quad r_2 = \frac{5 - \sqrt{21}}{2} = e^{-\varphi},\quad  \varphi = \mathrm{arccosh}\bigl(\tfrac{5}{2}\bigr).
        \end{align*}
    Enforcing the boundary conditions yields, after some simplification,
        \begin{align}\label{eq:closed-form-asym}
            (\Delta_k^{(3)})^{-1}_{ij}
            &=
            \begin{cases}
            \displaystyle
            \frac{\mathrm{sinh}(i\varphi)\,\mathrm{sinh}\bigl((k-j+1)\varphi\bigr)}
                {\mathrm{sinh}(\varphi)\,\mathrm{sinh}\bigl((k+1)\varphi\bigr)},
            & i \le j,\\
            \displaystyle
            \frac{\mathrm{sinh}(j\varphi)\,\mathrm{sinh}\bigl((k-i+1)\varphi\bigr)}
                {\mathrm{sinh}(\varphi)\,\mathrm{sinh}\bigl((k+1)\varphi\bigr)},
            & i > j.
            \end{cases}
        \end{align}

    Next we compute $(\widetilde{\Delta}_k^{(0)})^\dagger$. This can be done by invoking the machinery presented in~\cite{bendito2010generalized}, restated in detail in~\cref{lem:inverting-tridiagonal}, as follows. Let
        \begin{align*}
            (\omega_1, \omega_2, \dotsc, \omega_{k+1}, \omega_{k+2}) := \frac{1}{\sqrt{3k+2}} (1, \sqrt{3}, \sqrt{3}, \dotsc, \sqrt{3}, 1).
        \end{align*}
    Then for $j=1$, we have
        \begin{align*}
            \frac{1}{\omega_j}\left(c_j\omega_{j+1} + c_{j-1}\omega_{j-1}\right) &= \sqrt{3}\omega_2 = 3,
        \end{align*}
    and similarly for $j=k+2$. For $1 < j < k+2$, we have
        \begin{align*}
            \frac{1}{\omega_j}\left(c_j\omega_{j+1} + c_{j-1}\omega_{j-1}\right) &= \frac{1}{\sqrt{3}} (\sqrt{3} + \sqrt{3}) = 2,
        \end{align*}
    so that $(\omega_j)_{j=1}^{k+2}$ satisfy~\cref{eq:omega-js}. Now define
        \begin{align*}
            A_s &:= \sum_{j=1}^s \omega_j^2 = \frac{1}{3k+2}\begin{cases}
                1 &\text{ if }s=1,\\
                3s-2 &\text{ if }1 < s < k+1\\
                3k+2 &\text{ if }s = k+2 
            \end{cases},\quad 1\leq s\leq k+2,
        \end{align*}
    and symmetrically 
        \begin{align*}
            B_t &= A_{k+2} - A_{t} = \frac{1 + 3(k+1-t)}{3k+2},\quad 1\leq t\leq k+1,\\
            B_{k+2} &:= 0.
        \end{align*}
    Note next that, letting $(c_1, c_2,\dotsc, c_{k+1}) := (\sqrt{3}, 1, \dotsc, 1, \sqrt{3})$, it holds that for each $1\leq s \leq k+1$, $c_s\omega_s\omega_{s+1} = \frac{3}{3k+2}$. Therefore by~\cref{lem:inverting-tridiagonal}, it holds that for $1\leq i \leq j\leq k+2$,
        \begin{align*}
            (\widetilde{\Delta}_k^{(0)})^\dagger_{ij} = (\widetilde{\Delta}_k^{(0)})^\dagger_{ji} &= \frac{(3k+2)\omega_i \omega_j}{3} \left[
                \sum_{s=1}^{i-1} A_s^2
                +
                \sum_{t=i}^{k+1}
                B_t^2
                -
                \sum_{t=i}^{j-1}
                B_t
                \right]
        \end{align*}
    Next we compute, for $0\leq t \leq k+1$,
        \begin{align*}
            \sum_{s=1}^{t} A_s^2 &= \frac{1}{(3k+2)^2} \left(\underbrace{1^2 + 4^2 + 7^2 + \dotsc +(3(k+1) - 2)^2}_{t\text{ terms }}\right)\\
            &=\frac{1}{(3k+2)^2} \left( 9\frac{t(t+1)(2t+1)}{6}-12\frac{(t)(t+1)}{2} +4t\right) \\ %\sum_{s=1}^{t}\\ % 9s^2 - 12s + 4 \\
            &= \frac{t(6t^2 - 3t - 1)}{2(3k+2)^2},
        \end{align*}
    and similarly, for $1\leq t\leq k+2$, we have
        \begin{align*}
            \sum_{t=i}^{k+1}
            B_t^2 - \sum_{t=i}^{j-1}
            B_t
            &= \frac{1}{(3k+2)^2}\sum_{t=i}^{k+1}\bigl(3k+4-3t\bigr)^2
               \;-\; \frac{1}{3k+2}\sum_{t=i}^{j-1}\bigl(3k+4-3t\bigr)\\
            &= \frac{(k+2-i)\bigl(6(k+2-i)^2-3(k+2-i)-1\bigr)}{2(3k+2)^2}\\
            &\quad \;-\; \frac{(j-i)\bigl(2(3k+4)-3(i+j-1)\bigr)}{2(3k+2)}\\
            &= h_k(k+2-i)\;-\;c_k(i, j),
        \end{align*}
    where
        \begin{align*}
            h_k(x) &:= \frac{x(6x^2 - 3x - 1)}{2(3k+2)},\, c_k(i, j) := \frac{(j-i)}{2(3k+2)}\bigl(2(3k+4)-3(i+j-1)\bigr),\; x\in\mathbb{R},\; i, j\in\mathbb{Z}.
        \end{align*}
    Thus we have
        \begin{align*}
            (\widetilde{\Delta}_k^{(0)})^\dagger_{ij} = (\widetilde{\Delta}_k^{(0)})^\dagger_{ji} &= \frac{\omega_i \omega_j}{3}\left(h_k(i-1) + h_k(k+2-j) - c_k(i, j)\right),\quad 1\leq i\leq j\leq k+2.
        \end{align*}    
\end{proof}

\begin{lemma}\label{lem:resistance-formulas}
    Let $k\ge3$ and write
        \begin{align*}
            \varphi := \mathrm{arccosh}(5/2),\qquad s_j := \sinh(j\varphi),\quad j\in\mathbb{Z}.
        \end{align*}
    Then for each $0\le j\le k-1$, it holds that
        \begin{align*}
            r_{\mathrm{cycle}}(j;k)&=\frac{2s_{j+1}s_{(k-j)}}{s_1s_{k+1}},\\
            r_{\mathrm{path}}(j;k)&=\frac{1}{3} + \frac{2}{3}\frac{s_{j+1}s_{k-j}+ s_{j+2}s_{k-j-1}-2s_{j+1}s_{k-j-1}}{s_1s_{k+1}},\\
            r_{\mathrm{cap}}(k)&=\frac{1}{3} + \frac{2}{3}\frac{s_k}{s_{k+1}}.
        \end{align*}
\end{lemma}
  
\begin{proof}
    Recall that for any two vertices $u,v$ in $G_k$, it holds
        \begin{align*}
            r_{uv}
            &=(\mathbf{1}_u-\mathbf{1}_v)^\top\,\mathbf{L}(G_k)^\dagger\,(\mathbf{1}_u-\mathbf{1}_v),\\
            \mathbf{L}(G_k)^\dagger
            &=\widetilde U\,
            \mathrm{diag}\bigl((\widetilde\Delta_k^{(0)})^\dagger,\,
                           (\Delta_k^{(3)})^{-1},\,
                           (\Delta_k^{(3)})^{-1}\bigr)\,
            \widetilde U^\top.
        \end{align*}
    Hence if we set
        \begin{align*}
            w=\widetilde U^\top\bigl(\mathbf{1}_u-\mathbf{1}_v\bigr),
            \quad
            M^\dagger=\mathrm{diag}\bigl((\widetilde\Delta_k^{(0)})^\dagger,\,
                                      (\Delta_k^{(3)})^{-1},\,
                                      (\Delta_k^{(3)})^{-1}\bigr),
        \end{align*}
    then $r_{uv}=w^\top M^\dagger w$. Since
        \begin{align*}
            \widetilde{U}=\mathrm{diag}\bigl(I_2,\underbrace{U,\dots,U}_{k\text{ times}}\bigr),
        \end{align*}
    we have for any two vertices $u,v$ in $G_k$ the general formula
        \begin{align*}
            \widetilde U^\top&(\mathbf{1}_u-\mathbf{1}_v)=\\
            &\begin{cases}
              \displaystyle
              \mathbf{1}_1 - \bigl(0,0,\dots,0,\mathbf{u}_0(i),\mathbf{u}_1(i),\mathbf{u}_2(i),0,\dots,0\bigr)^\top,
              & (u,v)=(x,(i,0)),\\[1em]
              \displaystyle
              \mathbf{1}_2 - \bigl(0,0,\dots,0,\mathbf{u}_0(i),\mathbf{u}_1(i),\mathbf{u}_2(i)\bigr)^\top,
              & (u,v)=(y,(i,k-1)),\\[1em]
              \displaystyle
              \bigl(0,0,\,\dots,\,0,\underbrace{\mathbf{u}_0(i),\mathbf{u}_1(i),\mathbf{u}_2(i)}_{\text{block }j},0,\dots\bigr)^\top\\
              \;-\;
              \bigl(0,0,\,\dots,\,0,\underbrace{\mathbf{u}_0(i'),\mathbf{u}_1(i'),\mathbf{u}_2(i')}_{\text{block }j'},0,\dots\bigr)^\top,
              & u=(i,j),\;v=(i',j')
            \end{cases}
        \end{align*}
    along with the obvious modifications when the order of $u, v$ is reversed. We establish the three claims using this setup. First, in the case of cycle edges, let $u=(i,j)$ and $v=(i+1,j)$ for some $0\leq i \leq 2$ and $0\leq j \leq k-1$ fixed. Then we have that
        \begin{align*}
            \widetilde{U}^\top (\mathbf{1}_u-\mathbf{1}_v) &= \bigl(0,0,\,\dots,\,0,\underbrace{\mathbf{u}_0(i) - \mathbf{u}_0(i+1) ,\mathbf{u}_1(i)- \mathbf{u}_1(i+1),\mathbf{u}_2(i) - \mathbf{u}_2(i+1)}_{\text{block }j},0,\dots\bigr)^\top\\
            &= \bigl(0,0,\,\dots,\,0,\underbrace{0 ,\mathbf{u}_1(i)- \mathbf{u}_1(i+1),\mathbf{u}_2(i) - \mathbf{u}_2(i+1)}_{\text{block }j},0,\dots\bigr)^\top.
        \end{align*}
    with $(\mathbf{u}_1(i)-\mathbf{u}_1(i+1))^2+(\mathbf{u}_2(i)-\mathbf{u}_2(i+1))^2=2$. Therefore it holds that
        \begin{align*}
            r_{\mathrm{cycle}}(j;k)
            =2\,\bigl(\Delta_k^{(3)}\bigr)^{-1}_{\,j+1,j+1} =  \frac{2s_{j+1}s_{(k-j)}}{s_1s_{k+1}}.
        \end{align*}
    In the case of a path edge $\{u, v\}$ of the form $u=(i, j), v=(i, j+1)$ for $0\leq i \leq 2$ and $1\leq j \leq k-1$ fixed, we have
        \begin{align*}
            \mathbf{1}_u-\mathbf{1}_v
            =\sum_{s=0}^2 \mathbf{u}_s(i)\,\bigl(\mathbf{v}^{(s)}_j-\mathbf{v}^{(s)}_{j+1}\bigr),
        \end{align*}
    and thus that
        \begin{align*}
            r_{\mathrm{path}}(j;k)
            &=\frac{1}{3}\bigl[
                (\widetilde\Delta_k^{(0)})^\dagger_{\,j+2,j+2}
                +(\widetilde\Delta_k^{(0)})^\dagger_{\,j+3,j+3}
                -2(\widetilde\Delta_k^{(0)})^\dagger_{\,j+2,j+3}
              \bigr]
              \\
            &\qquad+\frac{2}{3}\bigl[
                (\Delta_k^{(3)})^{-1}_{\,j+1,j+1}
                +(\Delta_k^{(3)})^{-1}_{\,j+2,j+2}
                -2(\Delta_k^{(3)})^{-1}_{\,j+1,j+2}
              \bigr],
        \end{align*}
    from which the claim follows upon applying~\cref{lem:block-pinv}. Finally take $u=x$ and $v=(i,0)$. Then
        \begin{align*}
            \mathbf{1}_x-\mathbf{1}_{(i,0)}
            =\underbrace{\bigl(\mathbf{1}_x-\mathbf{u}_0(i)\mathbf{v}^{(0)}_0\bigr)}_{=:f^{(0)}}
             +\underbrace{\bigl(-\,\mathbf{u}_1(i)\mathbf{v}^{(1)}_0\bigr)}_{=:f^{(1)}}
             +\underbrace{\bigl(-\,\mathbf{u}_2(i)\mathbf{v}^{(2)}_0\bigr)}_{=:f^{(2)}},
        \end{align*}
    so
        \begin{align*}
            r_{\mathrm{cap}}(k)
            =((\widetilde\Delta_k^{(0)})^\dagger f^{(0)})^\top f^{(0)}
             +\sum_{s=1}^2\mathbf{u}_s(i)^2\,
              (\Delta_k^{(3)})^{-1}_{\,1,1}.
        \end{align*}
    We can compute
        \begin{align*}
            ((\widetilde\Delta_k^{(0)})^\dagger f^{(0)})^\top f^{(0)}
            =\frac{1}{3},
            \qquad
            (\Delta_k^{(3)})^{-1}_{\,1,1}
            =\frac{\sinh(k\varphi)}{\sinh\bigl((k+1)\varphi\bigr)},
        \end{align*}
    from which the claim follows.
\end{proof}

\begin{theorem}\label{thm:positivity-tube}
    For each $k\geq 1$, and $u\in V(G_k)$, the effective resistance curvature $\resistance{u}$ at vertex $u$ satisfies $\resistance{u} > 0$.
\end{theorem}

\begin{proof}[Proof of \cref{thm:positivity-tube}]
    The cases $k=1, 2$ can be handled via direct computation and are omitted. Assume $k\ge 3$. As a reminder we recall the definition of resistance curvature:
        \begin{align*}
            \resistance{u}\;=\;1-\frac12\sum_{v\sim u} r_{uv}.
        \end{align*}
    where $r_{uv}$ is the effective resistance between adjacent vertices $u$ and $v$ in $G_k$. Throughout, we write $C_j:=\cosh(j\varphi)$ for $j\in\mathbb{Z}$. Note the following identity for integers $a,b\in\mathbb{Z}$:
        \begin{align}\label{eq:prod-to-sum}
            2\,s_a s_b \;=\; C_{a+b}-C_{a-b},
        \end{align}
    and note additionally that the sequences $(C_n)$ and $(s_n)$ satisfy the linear recurrences
        \begin{align}\label{eq:recurrences}
            C_{n+1}=5\,C_n-C_{n-1},\qquad s_{n+1}=5\,s_n-s_{n-1}\qquad(n\in\mathbb Z),
        \end{align}
    as in the proof of~\cref{lem:block-pinv}. Set $D:=s_1 s_{k+1}>0$. First we consider the case of the cap vertices $x, y\in V(G_k)$. Each of $x,y$ has three neighbors and all three incident edges have resistance $r_{\mathrm{cap}}(k)$. Thus
        \begin{align*}
            \resistance{x}=\resistance{y}=1-\frac12\cdot 3\,r_{\mathrm{cap}}(k)
    =1-\frac32\left(\frac{1}{3}+\frac{2}{3}\cdot\frac{s_k}{s_{k+1}}\right)
    =\frac12-\frac{s_k}{s_{k+1}}.
        \end{align*}
    From~\cref{eq:recurrences} and the monotonicity $s_{n-1}<s_n$ for $n\ge1$, we have
        \begin{align*}
            s_{k+1}=5s_k-s_{k-1}\ge 5s_k-s_k=4s_k,
        \end{align*}
    hence $s_k/s_{k+1}\le 1/4$ and therefore $\resistance{x} = \resistance{y}= \ge \frac{1}{4} > 0$. Second, we consider the case of the interior vertices $(i,j)$, for $0\le i\le 2$ and $0\le j\le k-1$. Each such vertex has degree $4$. Note that by symmetry, the curvature $\resistance{(i, j)}$ does not depend on $i$. We treat separately the interior levels $1\le j\le k-2$ and then the boundary levels $j=0$ and $j=k-1$. For the interior levels, the neighbors of vertex $(i,j)$ are $(i\pm1,j)$ (two cycle edges) and $(i,j\pm1)$ (two path edges). Thus we have
        \begin{align*}
            \sum_{v\sim (i,j)} r_{(i,j),v}
    =2\,r_{\mathrm{cycle}}(j;k)+r_{\mathrm{path}}(j-1;k)+r_{\mathrm{path}}(j;k).
        \end{align*}
    Substituting the formulas provided in~\cref{lem:resistance-formulas} gives
        {\footnotesize \begin{align*}
            S_j &:=2\,r_{\mathrm{cycle}}(j;k)+r_{\mathrm{path}}(j-1;k)+r_{\mathrm{path}}(j;k)\\
                &=\frac{2}{3}+
                \frac{1}{D}\left[
                4\,s_{j+1}s_{k-j}
                +\frac{2}{3}\Big(2 s_{j+1}s_{k-j}+s_j s_{k-j+1}+s_{j+2}s_{k-j-1}
                -2 s_j s_{k-j}-2 s_{j+1} s_{k-j-1}\Big)
                \right].
        \end{align*}}
    Write
        \begin{align*}
            E_j :=8\,s_{j+1}s_{k-j}+ s_j s_{k-j+1}+ s_{j+2}s_{k-j-1} -2 s_j s_{k-j}-2 s_{j+1} s_{k-j-1}.
        \end{align*}
    We now simplify $E_j$ using~\cref{eq:prod-to-sum}. Let $t:=2j+1-k$. Then
        \begin{align*}
            E_j&=4\big(C_{k+1}-C_{t}\big)
            +\tfrac12\big(C_{k+1}-C_{t-1}\big)
            +\tfrac12\big(C_{k+1}-C_{t+2}\big)
            -\big(C_k-C_{t+1}\big)
            -\big(C_k-C_{t+2}\big)\\
            &=\big(4+\tfrac12+\tfrac12\big)C_{k+1}-2C_k
            +\big(C_{t+1}+C_{t+2}-4C_t-\tfrac12 C_{t-1}-\tfrac12 C_{t+2}\big).
        \end{align*}
    From~\cref{eq:prod-to-sum}, one has $C_{t-1}+C_{t+1}=2C_1C_t$ and $C_{t-2}+C_{t+2}=2C_2C_t$. Further, since $C_1=\cosh\varphi=5/2$ and $C_2=\cosh(2\varphi)=2C_1^2-1=23/2$, we get $2C_1-4-C_2=-21/2$. Therefore we have
        \begin{align*}
            E_j = 5C_{k+1}-2C_k - \frac{21}{2}\,C_t.
        \end{align*}
    On the other hand, by~\cref{eq:prod-to-sum}, $2D=2\,s_1 s_{k+1}=C_{k+2}-C_k$, and by~\cref{eq:recurrences}, $C_{k+2}=5C_{k+1}-C_k$, so $2D=5C_{k+1}-2C_k$. Hence
        \begin{align*}
            2D-E_j=\frac{21}{2}\,C_t.
        \end{align*}
    It follows that
        \begin{align*}
            \resistance{(i,j)}
    =\frac{1}{3D}\big(2D-E_j\big)
    ={\;\frac{7}{2\,s_1 s_{k+1}}\;\cosh\big((2j+1-k)\varphi\big)\;}.
        \end{align*}
    Since $\cosh(\cdot)\ge 1$, this yields $\resistance{(i,j)}>0$ for all interior $j$. The final case of the boundary levels $j=0$ and $j=k-1$ is similar and the claim follows.
\end{proof}

\end{document}